\documentclass[11pt, reqno]{amsart}
\usepackage{amsfonts,amsmath,amsthm}
\usepackage{amssymb,epsfig}
\usepackage{pgf,tikz}
\usepackage{fancybox}
\usepackage[T1]{fontenc}
\usepackage{stmaryrd}
\usepackage{graphicx}
\usepackage{eso-pic}
\usepackage[top=2cm, bottom=2cm, outer=0cm, inner=0cm, footskip=0.35in]{geometry}
\usepackage{fullpage}

\usepackage[unicode=true]{hyperref}
\hypersetup{colorlinks = true}
\hypersetup{
    colorlinks = true,
    linkcolor = {blue},
    citecolor={blue}
}

\usepackage[above,below,verbose,section]{placeins}
\makeatletter

\usepackage{mathrsfs}
\usepackage{mathpazo}
\usepackage{lastpage}
\usepackage[shortlabels]{enumitem}

\usepackage{tikz}
\usetikzlibrary{shapes.geometric, arrows, calc}
\tikzstyle{startstop} = [rectangle, rounded corners, minimum width=4cm, minimum height=2cm,text centered, draw=black, fill=red!30]
\tikzstyle{arrow} = [thick,->,>=stealth]
\makeatletter
\def\namedlabel#1#2{\begingroup
    #2%
    \def\@currentlabel{#2}%
    \phantomsection\label{#1}\endgroup
}
\makeatother

\usepackage[colorinlistoftodos,prependcaption]{todonotes}
\presetkeys%
{todonotes}%
{inline,backgroundcolor=white}{}
\tikzset{/tikz/notestyleraw/.append style={text=black}}

\newtheorem{thm}{Theorem}[section]

\newtheorem{lem}[thm]{Lemma}
\newtheorem{defn}[thm]{Definition}
\newtheorem{prop}[thm]{Proposition}
\newtheorem{cor}[thm]{Corollary}

\newtheorem{ex}[thm]{Example}
\newtheorem{rmk}[thm]{Remark}
\newcommand{\be}{\begin{eqnarray}}
\newcommand{\ee}{\end{eqnarray}}
\newcommand{\ben}{\begin{eqnarray*}}
\newcommand{\een}{\end{eqnarray*}}
\newcommand{\beal}{\begin{aligned}}
\newcommand{\enal}{\end{aligned}}
\newcommand{\beq}{\begin{equation}}
\newcommand{\eeq}{\end{equation}}

\newcommand{\eps}{\varepsilon}

\newcommand{\lb}{\lambda}

\newcommand{\T}{\mathbb{T}}
\newcommand{\R}{\mathbb{R}}

\newcommand{\N}{\mathbb{N}}

\newcommand{\Z}{\mathbb{Z}}

\newcommand{\om}{\omega}

\newcommand{\Dt}{\Delta}

\newcommand{\cC}{\mathcal{C}}

\newcommand{\cH}{\mathcal{H}}

	\def\textb{\textcolor{blue}}

\title{Quantitative homogenization for static contact Hamilton-Jacobi equations}

\thanks{{\it Statements and Declarations: }The authors declare no competing interests.}
\date{\today}

\numberwithin{equation}{section}

\begin{document}

\maketitle

\centerline{ Gengyu Liu$^\dagger$, Son N.T. Tu$^\ddagger$, and Jianlu Zhang$^*$}
\medskip

{\footnotesize
    \centerline{ $^{\dagger, *}$ State Key Laboratory of Mathematical Sciences,}
    \centerline{ Academy of Mathematics and Systems Science,}
    \centerline{ Chinese Academy of Sciences, Beijing 100190, China}
    \centerline{{\it Email: }$^\dagger$liugengyu@amss.ac.cn,\; $^*$jellychung1987@gmail.com}
}
\bigskip

{\footnotesize
    \centerline{ $^\ddagger$ Department of Mathematics, Baylor University}
    \centerline{ Waco, Texas 76708, USA}
    \centerline{{\it Email: }son\_tu@baylor.edu}
}
\bigskip

\begin{abstract} 
We characterize possible pairs $(u_\varepsilon,c)\in C(\mathbb{R}^n\backslash\varepsilon\mathbb{Z}^n,\mathbb{R})\times\mathbb{R}$ 
addressing the homogenization problem for Hamilton--Jacobi equations
\begin{equation*}
	H\left(\frac{x}{\varepsilon}, d u_\varepsilon, u_\varepsilon\right)=c, 
	\quad 
	\left({\mathrm resp.} \quad 
	H\left(\frac{x}{\varepsilon}, d u_\varepsilon, u_\varepsilon\right)=\varepsilon\Delta u_\varepsilon+c \right)	
\end{equation*}
for all $\varepsilon>0$. 
Under a (not necessarily strict) monotonicity assumption on the Hamiltonian, we proposed certain criteria (based on the structure of Mather measures), under which all possible 
solutions $u_\varepsilon$ converge to a uniquely identified limit $u\in C(\mathbb{R}^n,\mathbb{R})$ solving the effective equation 
\[
\overline H( du,u)=c,\quad  ({\mathrm resp.}\quad 
\overline H(du,u)=\Delta u+c)
\]
as $\varepsilon\rightarrow 0_+$ with a uniform rate $\mathcal{O}(\varepsilon)$. 
\end{abstract}


\section{Introduction}\label{s1}

Let $\T^n:=\R^n\slash\Z^n$ is equipped with the (quotient) Euclidean norm $|\cdot|$. 
The {\it contact Hamiltonian} $H:(x,p,\theta)\in \T^n\times\R^n\times\R\rightarrow \R$ is a continuous 
function satisfying the following assumptions:
\begin{itemize}
	\item[\namedlabel{itm:A1}{$(\mathcal{A}_1)$}] $H$ is convex in $p$ for every $(x,\theta)\in \T^{n} \times \R$.
	
	\item[\namedlabel{itm:A2}{$(\mathcal{A}_2)$}] $H$ is superlinear in $p$ for every $(x,\theta)\in \T^{n}\times \R$, i.e., $\lim_{\vert p \vert \rightarrow \infty} 	 \frac{H(x,p,\theta)}{|p|} = + \infty$.

    \item[\namedlabel{itm:A3}{$(\mathcal{A}_3)$}]  $\partial_{\theta}H(x,p,\theta) $ is continuous  and non-negative everywhere.
\end{itemize} 
We consider the homogenization problem associated with such a Hamiltonian. Namely, for any $\varepsilon>0$, 
we seek possible pairs $(u_\eps,c)\in C(\eps\T^n,\R)\times\R$ solving the {\it static Hamilton-Jacobi equations}:  
\begin{equation}\label{eq:HJ-ep}
	H\left(\frac{x}{\varepsilon}, d_x u, u\right) = c,\quad x\in\eps\T^n
\end{equation} 
and 
\begin{equation}\label{eq:HJ-ep-2nd}
	H\left(\frac{x}{\varepsilon}, d_x u, u\right) =\eps\Dt u+ c,\quad x\in\eps\T^n,
\end{equation} 
in the viscosity sense, and study the convergence of $u_\eps$ as $\eps\rightarrow 0_+$. The target is twofold. 
First,  we prove the existence of a non-empty connected set $\cC\subset\R$ such that equation \eqref{eq:HJ-ep} or \eqref{eq:HJ-ep-2nd} is solvable for any $c\in\cC$. Second, since there may exist multiple solutions (which differ even up to 
additive constants) associated with the same value $c\in\cC$, 
certain criteria will be proposed for the convergence of all possible $u_\eps$ 
solving \eqref{eq:HJ-ep} or \eqref{eq:HJ-ep-2nd} to a uniquely identified function  $u\in {\rm BUC}(\R^n,\R)$ 
solving the {\it effective} or {\it cell} problem
\begin{equation}\label{eq:HJ-overline}
    \overline{H}(d_x u, u) = c,\quad x\in\R^n.
\end{equation}
Here the {\it effective Hamiltonian} $\overline{H}: (p, \theta)\in\R^{n+1}\rightarrow\R$ is defined 
as the unique value that makes 
\beq\label{eq:erg-const-1st}
H(x, p + d_x u,\theta) = \overline{H}(p, \theta),\quad x\in\T^n
\eeq
or 
\beq\label{eq:erg-const-2nd}
	H(x, p + d_x u,\theta) = \Dt u+\overline{H}(p, \theta),\quad x\in\T^n
\eeq
is solvable (see \cite{lions_homogenization_1986}). 
Moreover, we can find a uniform constant $C:=C(H,c)>0$ such that 
\[
\|u_\eps-u\|_{L^\infty(\R^n)}\leq C \eps,\quad\forall\eps\in(0,1]. 
\]

A main feature of this paper is the extension of homogenization results to a broader class of equations without assuming uniqueness of solutions. A related setting arises for Hamiltonians that are $\left({u}/{\varepsilon}\right)$-periodic, a structure motivated by dislocation dynamics (see \cite{imbert_homogenization_2008, imbert_homogenization_2008', mitake_ni_tran_quantitative_2025} and the references therein).
Our approach relies on some new type of comparison principle (see Proposition \ref{prop:ss} and Proposition \ref{prop:cp-vis}) developed from the Aubry-Mather theory (for both 1st and 2nd order equations), which helps us to quantitatively constrain the multiple solutions $u_\eps$ and then 
constrain $\|u_\eps-u\|_{L^\infty(\R^n)}$ as $\eps\rightarrow 0_+$.

Another feature is the existence of continuous (indeed classical) solutions in the second-order case (see Section~\ref{s4}). Recall that the Perron method used in \cite{crandall_users_1992}  is unavailable in our framework due to the lack of a strict monotonicity of $H$ in $u-$variable, so we employ a vanishing discount procedure 
to bypass this probem. Although the idea of such a procedure is quite similar with \cite{DFIZ, IMT1,IMT2}, the involvement of $\eps-$parameter requires additional arguments for 
 the uniform boundedness and Lipschitz continuity of any family of solutions $\{u_\varepsilon\}_{\varepsilon>0}$ to \eqref{eq:HJ-ep-2nd} (see Proposition~\ref{prop:Bernstein}). 
Moreover, we uncover a global ordering structure among the multiple solutions $u_\varepsilon$ from the viewpoint of weak KAM theory. To the best of our knowledge, this ordering property has not been explicitly described in the existing literature. 

\medskip

The quantitative theory of homogenization has been extensively developed over the past two decades. For first-order equations with multiscale structure, the convergence rate \(O(\varepsilon^{1/3})\) was first obtained in \cite{capuzzo-dolcetta_rate_2001}. Using a similar approach, \cite{camilli_rates_2009} established a rate for multiscale homogenization of static fully nonlinear elliptic equations, and later derived a rate for homogenization with a vanishing viscosity process in \cite{camilli_homogenization_2016}. For single-scale homogenization of Cauchy problems for viscous Hamilton--Jacobi equations, \cite{qian_optimal_2024} recently proved the optimal \(O(\varepsilon^{1/2})\) rate.
In these works, the convexity of the Hamiltonian \(H\) is not required; however, the uniqueness of \(u_\varepsilon\) is typically essential, since the arguments rely on a comparison principle. Our setting is close in spirit to \cite{camilli_homogenization_2016,capuzzo-dolcetta_rate_2001}, but we do not assume uniqueness of \(u_\varepsilon\) (in particular, we do not impose strict monotonicity). Moreover, instead of the doubling-of-variables method used there, we adopt a variational approach with a stronger dynamical-systems flavor.

\medskip 

Finally, we would like to mention that the optimal rate of convergence $O(\eps)$ for single-scale homogenization of convex first-order equations was recently obtained in \cite{tran_optimal_2025}, and then $O(\eps^{1/2})$ for multi-scale case in \cite{han_rate_2023}, with $\mathcal{O}(\varepsilon)$ for 1D multi-scale case in \cite{tu_rate_2021}. 
We also refer the authors to \cite{han_tu_quantitative_2025,hu_polynomial_2025, mitake_ni_rate_2025_system, mitake_ni_quantitative_2025-Neumann} for more quantitative results on the homogenization in different contexts, e.g., the quasi-periodic setting, the weakly coupled systems, and the case with boundary conditions. We note that the existence of a pair of solutions to the first-order contact problem \eqref{eq:HJ-ep-2nd} has also been studied in \cite{jing_generalized_2020} in the absence of homogenization. In addition to a Perron-based method, a fixed-point argument yields an additional pair of solutions $(u,c)$.

\subsection{Main results} 
Our first result concerns the first-order equation \eqref{eq:HJ-ep}. 

\begin{thm}[First-order problem]\label{thm:main} 
Assume \ref{itm:A1}--\ref{itm:A3}.
\begin{enumerate}[(i)]
\item There exists an admissible set 
\begin{equation*}
	\cC:=\{\overline H(0,\theta)\in\R ~|~ \eqref{eq:erg-const-1st} \text{ admits a continuous solution for }p=0, \theta\in\R\}
\end{equation*}
such that for any $c\in\cC$ fixed, equation \eqref{eq:HJ-ep} and \eqref{eq:HJ-overline} are simultaneously solvable for all $\eps>0$. Moreover, $\mathcal{C}$ is a connected set in $\R$.  

\item Let $c\in \mathrm{int}\;\mathcal{C}$. 
Denote by $I(c):=\{\theta\in\R ~|~ \overline H(0,\theta)=c\}$.
If $\inf I(c) >- \infty$, the function 
\[
u^-_\eps:=\inf\{\om\in C(\eps\T^n,\R)~|~\om\text{ is a solution of } \eqref{eq:HJ-ep} \}
\]
is the minimal solution to \eqref{eq:HJ-ep}, and there exists a constant $C:=C(H,c)>0$ such that 
\begin{equation*}
	\Vert u_\eps^- - \inf I(c) \Vert_{L^\infty(\R^n)} \leq C \varepsilon\qquad\text{for}\;\eps\in(0,1].
\end{equation*}
Similarly, if  $\sup I(c) <+\infty$, the function 
\[
u_\eps^+:=\sup\{\om\in C(\eps\T^n,\R)~|~\om\text{ is a solution of } \eqref{eq:HJ-ep} \}
\]
is the maximal solution to \eqref{eq:HJ-ep} and 
\begin{equation*}
	\Vert u_\eps^+ - \sup I(c) \Vert_{L^\infty(\R^n)} \leq C \varepsilon\qquad\text{for}\;\eps\in(0,1]. 
\end{equation*}
\item 
For any $c \in \mathcal{C}$ such that $I(c)=\{\theta\}$ is a singleton, $u\equiv \theta$ is the unique solution of  equation \eqref{eq:HJ-overline} in ${\rm BUC}(\R^n,\R)$.  Consequently, there exists a $C:=C(H,c)>0$ such that for any solution $u_\eps$ of  equation \eqref{eq:HJ-ep}, $\Vert u_{\varepsilon} - u \Vert _{L^\infty(\R^n)} \leq C \varepsilon$ for $\eps\in(0,1]$.
\end{enumerate}
\end{thm}

\begin{rmk}\label{rmk:1st}
 As necessary explanation of Theorem \ref{thm:main}, some remarks are listed in order:
\begin{enumerate}[(i)]
\item If ${\rm int}\;\cC\neq\emptyset$, then for almost every $c\in\cC$, $I(c)$ is a singleton (see Lemma \ref{lem:I-single} for the proof). That implies Theorem \ref{thm:main}-(iii) is available for rather general values of $\cC$. However, there indeed exists an Hamiltonian $H(x,p,\theta)$ with non-trivial $\theta-$dependence, such that ${\rm int}\;\cC=\emptyset$ (see Example \ref{ex:sol1}). Such an exception stops us from considering the related  homogenization problem.
\item To make the solution $u_\eps$ of \eqref{eq:HJ-ep} be unique, we need certain {\bf ordinal Mather measure set} $\mathfrak M_-(u_\eps)$ be empty (see Definition \ref{defn:ord-M}). This condition is derived from a {\bf weak comparison principle}  established in Proposition \ref{prop:ss}. Unlike the classical comparison principle, our weak one only allows comparison between any subsolution and any solution of \eqref{eq:HJ-ep}; it does not permit replacing the solution with a supersolution.
\item When $I(c)$ is a singleton for some $c\in\cC$, we can easily get that ${\rm int}\;\cC\neq\emptyset$. Later, in Lemma \ref{lem:sgt}, we obtain that $\mathfrak M_-(\theta)=\emptyset$ is a sufficient condition for $I(c)$ being singleton, where $\overline H(0,\theta)=c$. 
If $I(c)$ is not a singleton, then \eqref{eq:HJ-overline} admits multiple solutions in ${\rm BUC}(\R^n,\R)$ (see Example \ref{ex:sol2}). So the associated homogenization problem is not well posed either. 
\end{enumerate}
\end{rmk}

Assumptions \ref{itm:A1}--\ref{itm:A3} are quite general. Under slightly stronger hypotheses, $I(c)$ can be related to the uniqueness of $u_\varepsilon$, as stated in the following corollary.

\begin{cor}\label{cor:homo}
Assume \ref{itm:A1}--\ref{itm:A3} and any one of the two:
\begin{itemize}
\item[\namedlabel{itm:A3'}{$(\mathcal{A}'_3)$}] $\partial_{\theta}H > 0$ for any $(x,p,\theta)\in \T^{n} \times \R^n \times \R$.
\item[\namedlabel{itm:A8}{$(\mathcal{A}''_3)$}]$H$ is convex in $\theta$ for every $(x,p) \in \T^n\times\R^n $.
\end{itemize}
For any $c \in \cC$ such that $I(c)=\{\theta\}$ is a singleton, there exists only one solution $u_\eps$ of \eqref{eq:HJ-ep}. Moreover, $u_\eps$ converges to $u\equiv \theta$ with the rate 
\[
\Vert u_{\varepsilon} - u \Vert_{L^\infty(\R^n)} \leq C\varepsilon
\]
for some constant $C=C(H,c)>0$. 
\end{cor}

Next, we present our quantitative homogenization result for equation \eqref{eq:HJ-ep-2nd}, which the following additional assumptions are made: 
\begin{itemize}
    \item[\namedlabel{itm:A4}{$(\mathcal{A}_4)$}] $H\in C^1(\T^n\times\R^n\times\R,\R)$.

    \item[\namedlabel{itm:A5}{$(\mathcal{A}_5)$}] $\partial_\theta H(x,p,\theta) \leq \rho^*$ for some constant $\rho^*\in (0,\infty)$ uniformly in $(x,p)\in \T^n\times \R^n$. 

    \item[\namedlabel{itm:A6}{$(\mathcal{A}_6)$}] There exist $m>1$ and constants $\Lambda_0 \in (0,1], M_0\geq 1$ such that, for every $p,q\in \R^n$ and $x,y\in \T^n$ we have
    \begin{equation*}
        \Lambda_0|p|^m - M_0 \leq H(x,p,0) . 
    \end{equation*}
    \item[\namedlabel{itm:A7}{$(\mathcal{A}_7)$}] For $\theta\in [-\ell, \ell]\subset \R$, there exists constants $\Lambda_\ell, M_\ell$ such that 
    \begin{equation*}
        |H(x,p,\theta) - H(y,p,\theta)| \leq (\Lambda_\ell|p|^m +M_\ell)|x-y|, \qquad x,y\in \T^n. 
    \end{equation*}
\end{itemize}

\begin{thm}[Second-order problem]\label{thm:main-2} 
Assume \ref{itm:A1}--\ref{itm:A7}.
\begin{enumerate}[(i)]
\item There exists an admissible set 
\begin{equation}\label{eq:C1def}
	\cC_1:=\{\overline H(0,\theta)\in\R ~|~ \eqref{eq:erg-const-2nd} \text{ admits a continuous solution for }p=0, \theta\in\R\}
\end{equation}
such that for any $c\in\cC_1$ fixed, equation \eqref{eq:HJ-ep-2nd} and \eqref{eq:HJ-overline} are simultaneously solvable for all $\eps>0$. 
\item  For any fixed $c \in \cC_1$, the solutions to \eqref{eq:HJ-ep-2nd} differ only by constants.
\item Denote by $I(c):=\{\theta\in\R| \overline H(0,\theta)=c\}$.
For any \textb{$c \in {\rm int}\cC_1$} such that $\inf I(c) >- \infty$, the function 
\[
u^-_\eps:=\inf\{\om\in C(\eps\T^n,\R)|\om\text{ is a solution of } \eqref{eq:HJ-ep-2nd} \}
\]
presents as the minimal solution to \eqref{eq:HJ-ep-2nd} and there exists a constant $C:=C(H,c)>0$ such that $\Vert u_\eps^- - \inf I(c) \Vert_{L^\infty(\R^n)} \leq C \varepsilon$  for $\eps\in(0,1]$. Similarly, if  $\sup I(c) < +\infty$, the function 
\[
u_\eps^+:=\sup\{\om\in C(\eps\T^n,\R)|\om\text{ is a solution of } \eqref{eq:HJ-ep-2nd} \}
\]
presents as the maximal solution to \eqref{eq:HJ-ep-2nd} and $\Vert u_\eps^+ - \sup I(c) \Vert_{L^\infty(\R^n)} \leq C \varepsilon$ for $\eps\in(0,1]$.
\item
 For any $c \in \mathcal{C}_1$ such that $I(c)=\{\theta\}$ is a singleton, $u\equiv \theta$ in the unique solution of  equation \eqref{eq:HJ-overline} in ${\rm BUC}(\R^n,\R)$.  Consequently, there exists a constant $C:=C(H,c)>0$ such that for any solution $u_\eps$ of  equation \eqref{eq:HJ-ep-2nd}, $\Vert u_{\varepsilon} - u \Vert_{L^\infty(\R^n)} \leq C \varepsilon$ for $\eps\in(0,1]$.
\end{enumerate}
\end{thm}

\begin{rmk} As in the first-order case, we make some remarks as follows.
\begin{enumerate}[(i)]
\item Theorem \ref{thm:main-2}-(ii) implies a global comparability of solutions of \eqref{eq:HJ-ep-2nd}. Not like the first-order case, the viscous term improves the Mather measures to be of Lebesgue type (see Lemma \ref{lem:b-con}), so accordingly the comparison principle becomes simpler. \item To make the solution $u_\eps$ of \eqref{eq:HJ-ep-2nd} be unique, we do not need to define certain ordinal Mather measures anymore. Instead, we propose a criterion \eqref{eq:uni-2nd} in Proposition \ref{prop:cp-vis}-(iv). We also propose a criterion \eqref{eq:sing-I(c)} in Lemma \ref{lem:sgt-3} to identify whether $I(c)$ is a singleton or not. Similarly to Remark \ref{rmk:1st}-(iii), $I(c)$ is a singleton implies that ${\rm int}\;\cC_1\neq\emptyset$.
\item If we replace the term $\eps\Dt u$ to a form $\eps\cdot{\rm tr}\left(A\left(\frac{x}{\varepsilon}\right)D^2u\right)$ with $A(\cdot):\T^n\rightarrow \mathbb S^{n\times n}$ being a positive definite matrix, Theorem \ref{thm:main-2} also holds. 
\end{enumerate}
\end{rmk}

At last, we show that certain strengthened assumptions can help to clarify the relation between $I(c)$ and the uniqueness of $u_\eps$:
\begin{cor}\label{cor:homo2}
Assume \ref{itm:A1}--\ref{itm:A7} and either \ref{itm:A3'} or \ref{itm:A8}. For any $c \in \cC_1$ such that $I(c)$ is a singleton, there exists only one solution $u_\eps$ of \eqref{eq:HJ-ep}. Moreover, $u_\eps$ converges to $u\equiv \theta$ with the rate 
\[
\Vert u_{\varepsilon} - u \Vert_{L^\infty(\R^n)} \leq C\varepsilon
\]
for some constant $C=C(H,c)>0$. 
\end{cor}
\noindent{\bf Organization of the paper.} In Sec. \ref{s2}, we figure out a new type comparison principle for solutions $u_\eps$ of \eqref{eq:HJ-ep}, via which we give the proof of  Theorem \ref{thm:main} and Corollary \ref{cor:homo}. In Sec. \ref{s3} we construct examples to explain all possible cases related with the criterion in Theorem \ref{thm:main}.
 In Sec. \ref{s4}, we present a generalized Perron method for \eqref{eq:HJ-ep-2nd} and some regularity results for the solutions.  In Sec. \ref{s5}, we give the comparison principle and prove  
 Theorem \ref{thm:main-2} and Corollary \ref{cor:homo2}.\\

\section{Solvability and Comparison Principle for First-Order Equations}\label{s2}

\subsection{Ergodic constants} 
By standard results in weak KAM  theory \cite{fathi_pde_2005}, under assumptions \ref{itm:A1}--\ref{itm:A3} and for any $(p,\theta)\in \R^n\times \R$ there exists a unique constant $\overline{H}(p,\theta) \in \R$ such that
\begin{align}\label{eq:1stHbar}
	H(y, p + d_yv(y), \theta) = \overline{H}(p,\theta) , \qquad y\in \T^n
\end{align}
can be solved with a viscosity solution $v\in C(\T^n)$. In fact, $\overline H (p,\theta)$ can be characterized as 
\begin{align}
\overline H (p,\theta) 
	& = \min_{\phi \in C^1(\T^n,\R)}\max_{x \in \T^n}H(x,p + d \phi,\theta)	\label{eq:minmaxHbar}\\
    &= \inf 
        \big\lbrace 
            c\in \R: \;\exists\;v\in \mathrm{Lip}(\T^n): H(x,p+Dv(x)) \leq c\;\text{in}\;\T^n\;\text{in the viscosity sense}
        \big\rbrace \nonumber 
\end{align}
due to \cite[Chapter 4]{tran_hamilton-jacobi_2021}.  
Let $u(x) := \varepsilon v\left(\frac{x}{\varepsilon}\right)$ for $x\in \varepsilon\T^n$, then  
\begin{equation}\label{eq:HJ-ep-theta}
	H\left(\frac{x}{\varepsilon}, p + d_x u, \theta\right) = \overline H (p,\theta),	\qquad x\in\eps\T^n
\end{equation} 
is solvable in the sense of viscosity for any $\eps>0$. The following result can be easily drawn:
\begin{lem}\label{lem:erg-1st} Assume \ref{itm:A1}--\ref{itm:A3}. Then $\overline H:(p,\theta)\in \R^{n+1} \rightarrow \R$ is continuous and non-decreasing in $\theta$. Consequently, 
\begin{equation*}
	\mathcal{C} := \{\overline{H}(0, \theta): \theta \in \R\}
\end{equation*}
is a connected set in $\R$. 
\end{lem}
\proof
 Due to \ref{itm:A3}, for any $a<b$ 
\[
H(x,d_x u_b,a)\leq H(x,d_x u_b,b)\leq c(b),\quad {\rm a.e. } \;x\in M
\]
where $u_b$ (resp. $u_a$) is a solution of \eqref{eq:erg-const-1st} with $(p,\theta)=(0,b)$ (resp. $=(0,a)$), therefore, $\overline H(0,a)\leq \overline H(0,b)$ since $\overline H(0,a)$ has to be the minimal value in $\R$ such that \eqref{eq:erg-const-1st} (associated with $(p,\theta)=(0,a)$) has a subsolution. For any  sequence $\R\ni a_n\to a$ as $ n\rightarrow+\infty$, we can find a sequence of solutions $\{u_n\}_{n\in\N}$. Due to \ref{itm:A2}, $\{u_n\}_{n\in\N}$ are uniformly Lipschitz, then uniformly bounded once we impose  $u_n(0)\equiv 0$ for any $n\in\N$. Suppose $c_*$ is an accumulating point of $\overline H(0,a_n)$ as $n\rightarrow+\infty$, then the associated subsequence of $\{u_n\}_{n\in\N}$ also has an accumulating function $\om$ as $n\rightarrow +\infty$, which is exactly a solution of  
\[
H(x, d_x\om, a)=c_*,\quad x\in \T^n.
\]
Due to \cite{fathi_pde_2005}, such a $c_*\in\R$ is unique, so $c_*=\overline H(0,a)$ and $\lim_{n\rightarrow+\infty} \overline H(0,a_n)=\overline H(0,a)$ follows. 
\qed

\begin{prop}\label{prop:erg-const}
For any $c \in \mathcal{C}$, the equation \eqref{eq:HJ-ep} and equation \eqref{eq:HJ-overline} are both solvable.
\end{prop}
\proof
For any $c \in \mathcal{C}$, there exists a $\theta \in \R$ such that $\overline{H}(0, \theta) = c$ and a viscosity solution $u_{\theta, \eps}$ of the equation \eqref{eq:HJ-ep-theta} with $p=0$ due to the weak KAM theory  \cite{fathi_weak_2008}. Then we can define two functions
\[
u_{\theta,\eps}^+ := \theta+u_{\theta,\eps}(x) + \Vert u_{\theta,\eps} \Vert_{L^\infty(\eps\T^n)}, \quad
u_{\theta,\eps}^-: = \theta+u_{\theta,\eps}(x) - \Vert u_{\theta,\eps} \Vert_{L^\infty(\eps\T^n)},
\]
which are supersolution and subsolution 
of equation \eqref{eq:HJ-ep} respectively. By the Perron method, there exists a solution $u_\eps$ of \eqref{eq:HJ-ep} satisfying $u_{\theta,\eps}^-\leq u_\eps\leq u_{\theta,\eps}^+$. Due to \ref{itm:A2}, such a $u_\eps$ is Lipschitz continuous.  As for 
equation \eqref{eq:HJ-overline}, $u \equiv \theta$ is exactly a solution due to the definition of $\overline H$. \qed

\subsection{Comparison principle and  Mather measures}\label{s2.2}

For any $\phi\in C(\T^n,\R)$, we can define a continuous Hamiltonian ${\bf H}_\phi:(x,p)\in \T^n\times\R^n\rightarrow\R$ by 
\[
{\bf H}_\phi(x,p):=H(x,p,\phi(x)).
\]
By the {\it Legendre transformation}, the associated Lagrange ${\bf L}_\phi:(x,v)\in \T^n\times\R^n\rightarrow\R$ can also be represented as 
\[
{\bf L}_\phi(x,v):=L(x,v,\phi(x))=\max_p\Big(\langle p,v\rangle-H(x,p,\phi(x))\Big).
\]
Since $H(x,p,u)$ satisfies \ref{itm:A1}--\ref{itm:A3}, then we can verify that ${\bf H}_\phi(x,p)$ is convex and superlinear in $p$ for any $x\in\T^n$. So we can apply weak KAM theory to ${\bf H}_\phi$ and get:

\begin{lem} \label{lem:weak-kam} Assume \ref{itm:A1}--\ref{itm:A3}. 
\begin{enumerate}[(i)]

\item There exists a unique {\bf ergodic constant} $c_\phi = \overline{\mathbf{H}}_\phi(0)\in\R$ such that 
\begin{equation}\label{eq:hj-phi}
	{\bf H}_\phi(x, du)=c_\phi,\quad x\in\T^n
\end{equation}
admits a solution. Moreover, $c_\phi$ is the minimal value such that the above equation admits subsolutions.

\item Let $\mathcal{H}$ be the set of \textbf{holonomic measures}, i.e., probability measures $\mu$ such that
\begin{equation*}
\int_{T\T^n} |v| \, d\mu < \infty 
\qquad \text{and} \qquad
\int_{T\T^n} \langle d_x \phi, v \rangle \, d\mu = 0
\quad \text{for all } \phi \in C^1(\T^n).
\end{equation*}
Then
\begin{equation*}
c_\phi = -\inf_{\mu \in \mathcal{H}} \int L \, d\mu.
\end{equation*}
The infimum is attained by the so-called \textbf{Mather measures} (associated with $\phi$), denoted by $\mathfrak{M}(\phi)$.

\item Any solution $u$ of \eqref{eq:hj-phi} is a weak KAM solution, in the sense that:  
\begin{itemize}
\item (Dominated)
for any $x,y \in \T^n$, any absolute continuous curves $\gamma:[0,t] \rightarrow \T^n$ with $\gamma(0) = x$ and $\gamma(t) = y$,
\begin{equation*}
	u(x) - u(y) \leq \int_0^t \big( 
	{\bf L}_\phi(\gamma(s),\dot{\gamma}(s)) + c_\phi\big) ds.
\end{equation*}
We denote by $u\prec{\bf L}_\phi+c_\phi$ if $u$ is a dominated function.
\item (Calibrated)
for every $x \in \T^n$, we can find a calibrated curve $\gamma:(-\infty,0] \rightarrow \T^n$ with $\gamma(0) = x$ such that for every $t \leq t' \leq 0$, we have
\[
u(\gamma(t')) - u(\gamma(t)) = \int_t^{t'} \big( {\bf L}_\phi(\gamma(s),\dot{\gamma}(s)) + c_\phi\big)  ds.
\]
\end{itemize}
\end{enumerate}
\end{lem}

\begin{defn}\label{defn:ord-M}
We call $\mu$ one {\bf ordinal Mather measure} associated with $\phi$, if $\mu$ is a Mather measure associated with $\mathbf{H}_\phi(x,p)$ such that 
\[
\int_{\T^n\times\R^n} \partial_u L(x,v,\phi(x)) d\mu=0.
\]
We denote $\mathfrak{M}_-(\phi)$ by the set of all  ordinal Mather measures associated with $\phi$.
\end{defn}

\begin{lem}[{\cite[Theorem 4.1]{chen_convergence_2024}}]\label{lem:comparision}
Assume \ref{itm:A1}--\ref{itm:A3}. Let $u,w : \T^n \to \R$ be such that
\begin{align*}
    \mathbf{H}_\varphi(x,d_x u) := H(x,d_x u,\varphi(x)) \leq c, 
    \qquad     
    \mathbf{H}_\psi(x,d_x w) := H(x,d_x w,\psi(x)) = c. 
\end{align*}
Then one of the following two holds:
\begin{enumerate}[(i)]
    \item  The maximum of $u - w$ can be attained at a point $x_{max} \in \T^n$ where $\varphi(x_{max}) \leq \psi(x_{max})$.

    \item We can find a Lipschitz curve $\gamma : (-\infty, 0] \rightarrow \T^n$ which is both $u$-calibrated for $\mathbf{L}_\phi + c$ and $w$-calibrated for $\mathbf{L}_\psi + c$, such that for all $t \in (-\infty, 0]$,
\[
u(\gamma(t)) - w(\gamma(t)) = \max(u - w), \quad \varphi(\gamma(t)) - \psi(\gamma(t)) > 0
\]
Moreover, in this case, denoting by $K < \infty$ a Lipschitz constant for the Lipschitz curve $\gamma: (-\infty, 0] \rightarrow \T^n$, we can find a holonomic measure $\mu \in \mathfrak M_-(\psi)$ of which ${\rm supp}(\mu)$ is contained in the compact subset $\{(x,v) \in \T^n: \vert v \vert_x \leq K\}$, such that
\[
\int_{T\T^n}[L(x,v,\varphi(x))+c]d\mu(x,v) = \int_{T\T^n}[L(x,v,\psi(x))+c]d\mu(x,v)
\]
and for all $(x,v) \in {\rm supp}(\mu)$
\[
u(x) - w(x) = \max_{\T^n}(u - w), \quad \varphi(x) - \psi(x) > 0
\]
\end{enumerate}
\end{lem}

\begin{lem}\label{lem:sgt}
For any $\theta \in I(c)$, if $\mathfrak{M}_-(\theta) = \varnothing$ then $I(c)$ is a singleton. 
\end{lem}

\proof
We know that 
\begin{equation*}
    \overline{H}(0, \theta) = -\inf_{\mu \in \cH}\int_{T\T^n} L(x,v,\theta) d\mu, 
\end{equation*}
and the infimum can be achieved by a holonomic measure $\mu$. If $\overline{H}(0,\theta_0) = \overline{H}(0,\theta) = c$ for some $\theta_0 \neq \theta$, then every value between $\theta$ and $\theta_0$ must also belong to $I(c)$. Indeed, there exists a holonomic measure $\mu_0$ such that 
\begin{align*}
    \int_{T\T^n} L(x,v,\theta_0) d\mu_0 = \int_{T\T^n} L(x,v,\theta) d\mu = -c.
\end{align*}
Without loss of generality, we assume $\theta_0 > \theta$. The monotonicity of $L$ in the variable $\theta$ guarantees that 
\begin{equation*}
     -c=- \overline{H}(0,\theta_0)= \inf_{\mu\in \mathcal{H}} \int_{T\T^n} L(x,v,\theta_0) d\mu \leq \int_{T\T^n} L(x,v,\theta_0) d\mu \leq \int_{T\T^n} L(x,v,\theta) d\mu \leq -c    
\end{equation*}
For any $\tilde{\theta} \in (\theta,\theta_0)$, we observe that 
\begin{equation*}
    -c=\int_{T\T^n} L(x,v,\theta_0)d\mu
    \leq 
    \int_{T\T^n} L(x,v,\tilde{\theta})d\mu  \leq \int_{T\T^n} L(x,v,\theta)d\mu = -c.
\end{equation*}
We obtain that 
\begin{equation*}
    \int_{T\T^n} \partial_u L(x,v,\theta)d\mu = 0
\end{equation*}
and thus $\mu \in \mathfrak{M}_-(\theta)$, which is a contradiction to $\mathfrak{M}_-(\theta) = \emptyset$. 
 Therefore, $\theta$ is the unique value in the set $I(c)$.\qed

\begin{lem}\label{lem:sgt-2}
For any $c \in \R$ such that $I(c)$ is a singleton, the 
solution $u\in {\rm BUC}(\R^n,\R)$ of equation \eqref{eq:HJ-overline} is unique and constant.
\end{lem}

\begin{proof}
Assume that $I(c)$ is a singleton. Then $u_1 \equiv \theta$ is a solution of \eqref{eq:HJ-overline}. Suppose there exists another solution $u_2 \in \mathrm{BUC}(\mathbb{R}^n,\mathbb{R})$, then $u_2(x + \zeta)\in {\rm BUC}(\R^n,\R)$ is also another solution to \eqref{eq:HJ-overline} for every $\zeta\in \R^n$. By stability of viscosity solution \cite{crandall_users_1992,tran_hamilton-jacobi_2021} we have 
\begin{equation*}
    \theta_+ := \sup_{x \in \R^n} u_2(x) \qquad\text{and}\qquad \theta_- := \inf_{x \in \R^n} u_2(x),
\end{equation*}
are subsolution and supersolution, respectively to \eqref{eq:HJ-overline}. Therefore, $\overline{H}(x,\theta_+) \leq c \leq \overline{H}(x,\theta_-)$.
By the monotonicity of $\overline{H}$ (inherited from that of $H$), we have $\overline{H}(x,\theta_+) \geq \overline{H}(x,\theta_-)$. 
Hence, $\overline{H}(x,\theta_+) = \overline{H}(x,\theta_-) = c$. 
Since $I(c)$ is a singleton, it follows that $\theta_+ = \theta_- = \theta$. Consequently, any other solution of \eqref{eq:HJ-overline} must coincide with the constant solution $\theta$.    
\end{proof}

\begin{rmk}
Using the similar argument as in Lemma \ref{lem:sgt-2}, we actually can get that $u := \inf I(c)$ (resp. $u:= \sup I(c)$) is the minimal (resp. maximal) solution of \eqref{eq:HJ-overline} in ${\rm BUC}(\R^n,\R)$ once $\inf I(c) > -\infty$ (resp. $\sup I(c) < \infty$).
\end{rmk}

In what follows, we apply aforementioned definitions and Lemmas to an $\eps-$dependent setting. Precisely, we denote 
\[
{\bf H}_{\phi_\eps}\left(\frac{x}{\varepsilon},p\right):= H\left(\frac{x}{\varepsilon}, p, \phi_\eps(x)\right),\quad (x,p)\in\eps\T^n\times\R^n
\]
 and the associated Lagrangian 
\[
{\bf L}_{\phi_\eps}\left(\frac{x}{\varepsilon},p\right):= L\left(\frac{x}{\varepsilon}, p, \phi_\eps(x)\right),\quad (x,p)\in\eps\T^n\times\R^n
\]
for any $\phi_\eps\in C(\eps\T^n,\R)$.
We obtain the associated Mather measure (resp. ordinal Mather measure) set by $\mathfrak M(\phi_\eps)$ (resp. $\mathfrak M_-(\phi_\eps)$). In fact, such a parametrization successfully extends the case $\eps=1$ to $\eps>0$. In particular, if we choose $\phi_\eps\in C(\eps\T^n,\R)$ being a solution (or subsolution) of \eqref{eq:HJ-ep}, the following comparison principle can be drawn:

\begin{prop}[Comparison principle]\label{prop:ss}
Assume \ref{itm:A1}--\ref{itm:A3}.
\begin{enumerate}[(i)]

\item Suppose $u^1_\eps$ (resp. $u^2_\eps$) is a solution (resp. subsolution) of \eqref{eq:HJ-ep} satisfying
\begin{equation}\label{eq:measure-order}
    \int_{\eps\T^n\times\R^n} u^1_\eps(x)d\mu(x,v) \geq \int_{\eps\T^n\times\R^n} u^2_\eps(x)d\mu(x,v) \ \ \ \ \forall \mu \in \mathfrak{M}_-(u^1_\eps)
\end{equation}
then $u^1_\eps \geq u^2_\eps$.
\item
if $\mathfrak{M}_-(u^1_\eps) = \varnothing$, then $u^1_\eps$ is the unique solution (on $\eps \T^n$) of \eqref{eq:HJ-ep}.
\item
If  $u_\eps$ is a solution of \eqref{eq:HJ-ep} and $w_\eps$ is a strict subsolution of \eqref{eq:HJ-ep}, i.e. $H(\frac{x}{\varepsilon}, d_x w_\eps, w_\eps) < c$ in the viscous sense, then  $w_\eps \leq u_\eps$.
\item
If $\sup I(c) < \infty$, then we can find a continuous supersolution $\nu_\eps$ of \eqref{eq:HJ-ep} which is not less than than any solution $u_\eps$ of \eqref{eq:HJ-ep}.
\end{enumerate}
\end{prop}

\begin{proof} \quad 

{\it (i).} We prove by contradiction. Assume that $u^1_\eps$ and $u^2_\eps$ are, respectively, a solution and a subsolution of \eqref{eq:HJ-ep}, satisfying 
\[
\int_{\eps\T^n \times \R^n} u^1_\eps(x)d\mu(x,v) \geq \int_{\eps\T^n \times \R^n} u^2_\eps(x)d\mu(x,v) \ \ \ \ \forall \mu \in \mathfrak{M}_-(u^1_\eps)
\]
In this case, $\mathfrak{M}_-(u^1_\eps)$ is a set of probability measures on $\eps \T^n \times \R^n$.
If there is some $x_0 \in \eps\T^n$ such that $u^2_\eps(x_0) - u^1_\eps(x_0) > 0$, then by Lemma \ref{lem:comparision}, there exists a measure $\mu\in\mathfrak{M}_-(u^1_\eps)$ such that
\begin{equation*}
    u^2_\eps - u^1_\eps \equiv \max(u^2_\eps - u^1_\eps) > 0 \qquad\text{on}\;\mathrm{supp}(\mu). 
\end{equation*}
Therefore 
\begin{equation*}
    \int_{\varepsilon\T^n\times \R^n} u^2_\varepsilon(x)d\mu(x,v) > 
    \int_{\varepsilon\T^n\times \R^n} u^1_\varepsilon(x)d\mu(x,v), 
\end{equation*}
which contradicts the assumption \eqref{eq:measure-order}.
\medskip 

{\it (ii).}
If $\mathfrak{M}_-(u^1_\eps) = \varnothing$, then \eqref{eq:measure-order} holds for any subsolution; consequently, $u^1_\eps \geq u^2_\eps$ for any other solution $u^2_\eps$ of \eqref{eq:HJ-ep}. Let $u^2_\varepsilon$ be such another solution of \eqref{eq:HJ-ep} and assume that $u^2_\eps(x_0) - u^1_\eps(x_0) > 0$ for some $x_0\in\eps\T^n$. Then, similar to (i), by Lemma \ref{lem:comparision}, there exists a measure $\mu\in\mathfrak{M}_-(u^1_\eps)$ such that
\begin{equation*}
    u^2_\eps - u^1_\eps \equiv \max(u^2_\eps - u^1_\eps) > 0 \qquad\text{on}\;\mathrm{supp}(\mu), 
\end{equation*}
which contradicts with $\mathfrak M_-(u^1_\eps)=\emptyset$.
\medskip

{\it (iii).}
Let us define the Hamiltonians 
\begin{align*}
{\bf H}_{u_\eps}\left(\frac{x}{\varepsilon},p\right):= H\left(\frac{x}{\varepsilon}, p, u_\eps(x)\right),\quad
{\bf H}_{w_\eps}\left(\frac{x}{\varepsilon},p\right):= H\left(\frac{x}{\varepsilon}, p, w_\eps(x)\right),
\qquad (x,p)\in\eps\T^n\times\R^n. 
\end{align*}
Assume the contrary that, there exists $x_0$ such that 
\begin{equation*}
    w_\eps(x_0) - u_\eps(x_0) = \max_{\varepsilon \T^n} (w_\varepsilon - u_\varepsilon) > 0. 
\end{equation*}
By Lemma \ref{lem:comparision}, there exists a measure $\mu\in\mathfrak M_-(u_\eps)$ such that 
\begin{align*}
    w_\varepsilon(x) - u_\varepsilon(x) = \max_{\varepsilon \T^n} (w_\varepsilon - u_\varepsilon) \qquad (x,v)\in \mathrm{supp}(\mu). 
\end{align*}
Due to \cite[Theorem 7.8-(ii)]{fathi_pde_2005}, the critical subsolution $w_\varepsilon$ is differentiable on $\pi_1(\mathrm{supp}(\mu))$, where $\pi_1$ is the projection onto the first coordinate. Hence $Dw_\varepsilon = Du_\varepsilon$ on $\pi_1(\mathrm{supp}(\mu))$, and thus there are points in $\varepsilon\T^n$ such that
\begin{equation*}
    {\bf H}_{w_\eps}\left(\frac{x}{\varepsilon},d_x w_\eps(x)\right)=c. 
\end{equation*}
That contradicts the fact that $w_\eps$ is a strict subsolution of \eqref{eq:HJ-ep} and leads to $w_\eps \leq u_\eps$.

\medskip 

{\it (iv).}
Since $\sup I(c) < \infty$, we can define $\theta_+ = \sup I(c)$. Then $\overline H(0,\theta_+) = c$ and $\theta \leq \theta_+$ for all $\theta \in I(c)$. We apply Lemma \ref{lem:comparision} for the Hamiltonians
\begin{equation*}
    {\bf H}_{u_\eps}\left(\frac{x}{\varepsilon},p\right):= H\left(\frac{x}{\varepsilon}, p, u_\eps(x)\right) \quad \text{and}\quad  
    {\bf H}_{\theta_+}\left(\frac{x}{\varepsilon},p\right):= H\left(\frac{x}{\varepsilon}, p, \theta_+\right), \qquad 
    (x,p)\in\eps\T^n\times\R^n    .
\end{equation*}
Let $w_{\theta_+}$ be a solution to $H(x,d_xw_{\theta_+}, \theta_+)  =c$ in $\T^n$, and  
\begin{equation}\label{eq:nu-eps}
    \nu_\eps(x) = \theta_+ + \eps w_{\theta_+} \left(\frac{x}{\eps}\right) + 2\eps\Vert w_{\theta_+}\Vert_{L^\infty(\T^n)}, \qquad x\in \varepsilon\T^n.
\end{equation}
Then $\nu_\varepsilon \geq \theta_+$ in $\varepsilon \T^n$, therefore
\begin{equation*}
    {\bf H}_{\nu_\varepsilon}\left(\frac{x}{\varepsilon}, d_x\nu_\varepsilon\right)  =   H\left(\frac{x}{\varepsilon}, d_x \nu_\varepsilon, \nu_\varepsilon \right) \geq c
    \qquad\text{and}\qquad 
     {\bf H}_{\theta_+}\left(\frac{x}{\varepsilon},d_x\nu_\varepsilon\right) = c 
     \qquad\text{in}\;\varepsilon \T^n. 
\end{equation*}
If $u_\varepsilon$ is a solution to \eqref{eq:HJ-ep}, we show that $\nu_\varepsilon \geq u_\varepsilon$. Indeed, we have
\begin{equation*}
    {\bf H}_{u_\varepsilon}\left(\frac{x}{\varepsilon}, d_xu_\varepsilon\right) = c \qquad\text{and}\qquad 
    {\bf H}_{\theta_+}\left(\frac{x}{\varepsilon},d_x\nu_\varepsilon\right) = c . 
\end{equation*}
If $(u_\varepsilon-\nu_\varepsilon)(x_0) >0$ for some $x_0\in \varepsilon\T^n$, then $u_\varepsilon(x_0) > \nu_\varepsilon(x_0) > \theta_+$. By Lemma \ref{lem:comparision}, then there exists a measure $\mu\in \mathfrak{M}_-(\nu_\varepsilon)$ such that 
\begin{equation*}
    u_\varepsilon(x) - \nu_\varepsilon(x) =\max_{\varepsilon\T^n}(u_\varepsilon-\nu_\varepsilon) , \qquad u_\varepsilon(x) - \theta_+ >0 \qquad\text{for}\;x\in \pi_1(\mathrm{supp}(\mu))
\end{equation*}
and
\begin{equation*}
    \int_{\varepsilon\T^n\times \R^n} L\left(\frac{x}{\varepsilon},v, u_\varepsilon(x)\right)\;d\mu(x,v) = \int_{\varepsilon\T^n\times \R^n} L\left(\frac{x}{\varepsilon},v, \theta_+\right)\;d\mu(x,v) = -c.
\end{equation*}
In this case, as $\mathrm{supp}(\mu)$ is compact, we can find a value $\theta'\in \R$ such that $u_\varepsilon(x) \geq \theta' > \theta_+$ for $x\in \pi_1(\mathrm{supp}(\mu))$. As a consequence, we have 
\begin{align*}
\int_{\varepsilon\T^n\times \R^n} L\left(\frac{x}{\varepsilon}, v, u_\varepsilon(x)\right)d\mu(x,v)
\leq 
    \int_{\varepsilon\T^n\times \R^n} L\left(\frac{x}{\varepsilon}, v, \theta'\right)d\mu(x,v) \leq \int_{\varepsilon\T^n\times \R^n} L\left(\frac{x}{\varepsilon}, v, \theta_+\right)d\mu(x,v) 
\end{align*}
and thus 
\begin{equation*}
    \int_{\varepsilon\T^n\times \R^n} L\left(\frac{x}{\varepsilon}, v, \theta'\right)d\mu(x,v) = -c.
\end{equation*}
This is a contradiction to the fact that $\overline{H}(0,\theta') > c$ as $\theta'=\theta_+ = \sup I(c)$. We therefore obtain the conclusion that $u_\varepsilon\leq \nu_\varepsilon$ on $\T^n$ for any $u_\eps$ the solution of \eqref{eq:HJ-ep}.
\end{proof}

\begin{lem}\label{lem:infu}
Assume \ref{itm:A1}--\ref{itm:A3}. Let $c\in \mathrm{int}\;\mathcal{C}$.
Define
\begin{align*}
    u_\varepsilon^- &:= \inf \left\{ w \in C(\varepsilon \mathbb{T}^n,\mathbb{R}) \;\middle|\; w \text{ is a solution of } \eqref{eq:HJ-ep} \right\}
    \quad \text{if } \inf I(c) > -\infty, \\
    u_\varepsilon^+ &:= \sup \left\{ w \in C(\varepsilon \mathbb{T}^n,\mathbb{R}) \;\middle|\; w \text{ is a solution of } \eqref{eq:HJ-ep} \right\}
    \quad \text{if } \sup I(c) < \infty.
\end{align*}
Then $u_\varepsilon^\pm$ are well defined, and continuous (thanks to a priori estimate given by coercivity) in the corresponding cases
$\inf I(c) > -\infty$ and $\sup I(c) < \infty$, respectively, and are viscosity solutions of \eqref{eq:HJ-ep}. Consequently, if $I(c)$ is a singleton, all solutions $u_\varepsilon$ of \eqref{eq:HJ-ep} are uniformly bounded for $0<\varepsilon<1$.
\end{lem}

\begin{proof} We consider the case of $u_\varepsilon^-$ first. \medskip

\noindent 
\underline{{\it Step 1. Well-posedness of $u^-_\varepsilon$. }}
Since $\inf I(c) > -\infty$, let us define $\theta_- = \inf I(c)$.
For any $\delta \in (0,1)$, we have $\theta_\delta = \theta_--\delta \notin I(c)$, and thus $\overline{H}(0,\theta_\delta) = c_\delta<c$. 
For $\delta>0$, let $w_\delta$ be the solution of 
\begin{equation*}
    H(x,D w_\delta(x), \theta_\delta) = c_\delta 
    \quad \text{in}\;\T^n \qquad\text{such that}\qquad w_\delta(0) =0. 
\end{equation*}
Since $c_\delta < c$ and $\theta_\delta \geq \theta_--1$, the coercivity \ref{itm:A2} implies that the family $\{w_\delta\}_{\delta\in (0,1)}$ is equi-Lipschitz, namely, $\Vert Dw_\delta\Vert_{L^\infty(\R^d)} \leq M$ for $M = M(c)$ (as $\theta_-$ also depends on $c$). Since $w_\delta(0) = 0$, we obtain
\begin{equation*}
    \Vert Dw_\delta\Vert_{L^\infty(\R^d)} + \Vert w_\delta\Vert_{L^\infty(\R^d)} \leq M
\end{equation*}
for some $M=M(c)$ independent of $\delta \in (0,1)$. We have
\begin{align*}
    & H\left(\frac{x}{\varepsilon}, Dw_\delta\left(\frac{x}{\varepsilon}\right), \theta_\delta\right) = c_\delta < c \quad \text{in}\; \varepsilon \T^n \\
    & \vartheta_{\delta, \varepsilon}^-(x) = \theta_\delta + \varepsilon w_\delta\left(\frac{x}{\varepsilon}\right)- \varepsilon \Vert w_\delta\Vert_{L^\infty(\T^n)} \qquad\text{is a strict subsolution of}\; \eqref{eq:HJ-ep}.
\end{align*}

\noindent 
By Proposition \ref{prop:ss}-(iii), any solution $u_\varepsilon$ to \eqref{eq:HJ-ep} satisfies
\begin{equation*}
    u_\varepsilon(x) \geq  \vartheta_{\delta, \varepsilon}^-(x) =  \theta_\delta + \varepsilon w_{\delta}\left(\frac{x}{\varepsilon}\right)
    - \varepsilon \| w_{\delta} \|_{L^\infty(\mathbb{T}^n)} \geq \theta_\delta - 2M.
\end{equation*}
Let $\delta\to 0_+$ we deduce that 
\begin{equation}\label{eq:below}
    u_\varepsilon(x) \geq \theta_- -2M. 
\end{equation}
In particular, $u_\varepsilon$ is uniformly bounded from below, and
\begin{align*}
    H\left(\frac{x}{\varepsilon}, Du_\varepsilon(x), \theta_- -2M\right) \leq H\left(\frac{x}{\varepsilon}, Du_\varepsilon(x),u_\varepsilon(x)\right) = c \qquad\text{in}\;\T^n.
\end{align*}
Moreover, the coercivity of $H$ implies that $\{u_\varepsilon\}_{0<\varepsilon<1}$ is equi-Lipschitz. This implies that $u_\eps^-$ is a well-defined Lipschitz function, and by property of viscosity solution $u^-_\varepsilon$ is a supersolution of \eqref{eq:HJ-ep} (see \cite[Lemma 4.2]{crandall_users_1992}). \medskip 

\noindent 
\underline{{\it Step 2. $u^-_\varepsilon$ is a solution of \eqref{eq:HJ-ep}. }}
On another had, by Lemma \ref{lem:weak-kam}-(iii), any solution $u_\varepsilon$ to \eqref{eq:HJ-ep} satisfies $u_\varepsilon \prec {\bf L}_{u_\varepsilon} + c$, i.e., for all $x,y\in \varepsilon \T^n$, $t>0$ and $\gamma\in \mathrm{AC}([0,t];\T^n)$ with $\gamma(0) = x, \gamma(t)=y$ then
\begin{equation*}
    u_\eps(y) - u_\eps(x) \leq 
    \int_0^t
    \Big( 
    {\bf L}_{u_\eps}\left(\frac{\gamma(s)}{\varepsilon},\dot{\gamma}(s)\right) + c\Big)  ds
    \leq 
    \int_0^t
    \Big( 
    {\bf L}_{u^-_\eps}\left(\frac{\gamma(s)}{\varepsilon},\dot{\gamma}(s)\right) + c\Big)  ds
\end{equation*}
where we use $u^-_\eps \leq u_\eps$ in the last inequality. This implies that 
\begin{align*}
    u^-_\eps(y) - u^-_\eps(x) \leq 
    \int_0^t
    \Big( 
    {\bf L}_{u^-_\eps}\left(\frac{\gamma(s)}{\varepsilon},\dot{\gamma}(s)\right) + c\Big)  ds
\end{align*}
for all $x,y\in \T^n$ and $\gamma\in \mathrm{AC}([0,t];\T^n)$ with $\gamma(0) = x, \gamma(t)=y$ as well. In other words, $u_\varepsilon^-$ is also a subsolution to \eqref{eq:HJ-ep} and thus it is a solution to \eqref{eq:HJ-ep}. \medskip 

\noindent 
\underline{{\it Step 3. Rate of convergence for $u^-_\varepsilon - \theta^-$. }} Let $w_{\theta_-}$ be a solution to 
\begin{equation*}
    H(x,Dw_{\theta_-}, \theta_-) = c \qquad\text{in}\;\T^n \qquad\text{with}\qquad w_{\theta_-}(0) = 0,
\end{equation*}
where $\theta_- = \inf I(c)$. By coercivity we can obtain 
\begin{equation*}
     \Vert Dw_{\theta_-}\Vert_{L^\infty(\R^d)} + \Vert w_{\theta_-}\Vert_{L^\infty(\R^d)} \leq M
\end{equation*}
for some $M=M(c)$ independent of $\delta \in (0,1)$.
Then
\begin{equation*}
    \theta_- + \varepsilon w_{\theta_-}\left(\frac{x}{\varepsilon}\right) \pm \varepsilon \|w_{\theta_-}\|_{L^\infty(\mathbb{T}^n)}
\end{equation*}
are Lipschitz sub- and supersolutions to \eqref{eq:HJ-ep}, respectively.
The Perron method guarantees the existence of a solution $u_\eps$ satisfying that 
\begin{equation}\label{eq:above}
    \theta_- + \varepsilon w_{\theta_-}\left(\frac{x}{\varepsilon}\right) - \varepsilon\Vert w_{\theta_-} \Vert_{L^\infty(\T^n)} 
        \leq u_\eps \leq 
    \theta_- + \varepsilon w_{\theta_-}\left(\frac{x}{\varepsilon}\right) + \varepsilon\Vert w_{\theta_-} \Vert_{L^\infty(\T^n)}. 
\end{equation}
From \eqref{eq:below}, \eqref{eq:above}, and the fact that $u^-_\varepsilon\leq u_\varepsilon$ we obtain 
\begin{equation}\label{eq:infu_eps}
    \| u^-_\eps - \theta_- \|_{L^\infty(\R^n)} \leq C\eps, \qquad C = 2\max \left\lbrace 
    \Vert w_{\theta_-} \Vert_{L^\infty(\T^n)}, M \right\rbrace . 
\end{equation}
\smallskip

Now we consider the case of $u_\varepsilon^+$. As $c\in \mathrm{int}\;\mathcal{C}$, we can find $\kappa>0$ such that $(c-\kappa, c+\kappa)\subset\mathcal{C}$. Let $c_0\in (c,c+\kappa)\subset \mathcal{C}$. By \eqref{eq:nu-eps} in Proposition \ref{prop:ss}-(iv), we can find a supersolution $\nu_{\varepsilon, c_0}$ with $\theta \in I(c_0)$ such that $\nu_{\varepsilon,c_0} \geq u$ for any subsolution $u$ to \eqref{eq:HJ-ep} with the value $c$.

If $\theta_+=\sup I(c) < \infty$ then for each $c'\in(c,c+\kappa)$, we must have $\inf I(c') > -\infty$. Indeed, if $\theta'\in I(c')$ then $\overline{H}(0,\theta') = c' > c = \overline{H}(0,\theta_+)$ and thus $\theta' \geq \theta_+$. Let $u^-_{\varepsilon, c'}$ be the minimal solution to \eqref{eq:HJ-ep} as in the previous part.
Let us define  
\begin{equation*}
    u_\eps^*:= \inf \big\lbrace u_{\eps,c'}^- ~|~ c' \in( c,c_0) \big\rbrace \leq \nu_{\varepsilon, c_0}. 
\end{equation*}
Hence $u_\eps^*$ is well defined. 
Since $\theta_+=\sup I(c) < \infty$, we denote by $w_{\theta_+}$ the solution to $ H(x,Dw_{\theta_+}, \theta_+) = c$ in $\T^n$ with $w_{\theta_+}(0) = 0$. 
Let 
\begin{align*}
    \vartheta^-_{\theta_+,\varepsilon}(x) = \theta_+ + \varepsilon w_{\theta_+}\left(\frac{x}{\varepsilon}\right) - \varepsilon \Vert w_{\theta_+}\Vert_{L^\infty(\T^n)} .
\end{align*}
For each $c'>c$, we have
\begin{align*}
    & H\left(\frac{x}{\varepsilon}, Du^-_{\varepsilon, c'}(x), u^-_{\varepsilon, c'}(x)\right) = c' \qquad\text{in}\;\T^n \\
    & H\left(\frac{x}{\varepsilon}, D\vartheta^-_{\theta_+,\varepsilon}(x), \vartheta^-_{\theta_+}(x)\right) \leq c \qquad\text{in}\;\T^n.
\end{align*}
Since $c'>c$, by Proposition \ref{prop:ss}-(iii), we have
\begin{equation*}
    \vartheta^-_{\theta_+,\varepsilon}(x) =  \theta_+ + \varepsilon w_{\theta_+}\left(\frac{x}{\varepsilon}\right) - \varepsilon \Vert w_{\theta_+}\Vert_{L^\infty(\T^n)} \leq u^-_{\varepsilon, c'}(x). 
\end{equation*}
As a consequence, we have
\begin{equation*}
    u_\eps^*(x) \geq 
    \vartheta^-_{\theta_+,\varepsilon}(x) =  \theta_+ + \varepsilon w_{\theta_+}\left(\frac{x}{\varepsilon}\right) - \varepsilon \Vert w_{\theta_+}\Vert_{L^\infty(\T^n)}. 
\end{equation*}
With the Proposition \ref{prop:ss}-(iii), if  $u$ is a subsolution to \eqref{eq:HJ-ep}, i.e.,
\begin{align*}
    H\left(\frac{x}{\varepsilon}, Du, u\right) \leq c \qquad\text{in}\;\T^n 
\end{align*}
then since $c'>c$, we have $u_{\eps,c'}^- \geq u$. 
Therefore $u_\eps^*$ is larger than any solution to \eqref{eq:HJ-ep}, and thus $u^*_\varepsilon\geq u^+_\varepsilon$. On the other hand, we also have $u^*_\varepsilon$ is a supersolution to \eqref{eq:HJ-ep}. By the same arguments as in proving $u_\eps^-\prec\mathbf{L}_{u_\eps} + c$, we can prove that $u_\eps^*\prec\mathbf{L}_{u_\eps^*} + c$ and hence a subsolution to \eqref{eq:HJ-ep}. Therefore, $u_\eps^*$ is a solution to \eqref{eq:HJ-ep} and thus $u_\eps^* \leq u_\eps^+$. We therefore conclude that 
\begin{equation}\label{eq:supu_eps}
    \| u^+_\eps - \theta_+ \|_{L^\infty(\R^n)} \leq C\eps, \qquad C = 3\Vert w_{\theta_+} \Vert_{L^\infty(\T^n)}. 
\end{equation}
As for $I(c)$ is a singleton, we have $\theta = \theta_+=\theta_-$. Therefore, for any solution $u_\eps$, we have $u_\eps^- \leq u_\eps \leq u_\eps^+$ and $\Vert u_\eps - \theta \Vert_{L^\infty(\R^n)} \leq C\eps$.
\end{proof}

We are now ready to give the proof of Theorem \ref{thm:main}.

\begin{proof}[Proof of Theorem \ref{thm:main}] Proof of (i) follows from Proposition \ref{prop:erg-const}, while (ii) follows from Lemma \ref{lem:infu} by \eqref{eq:infu_eps} and \eqref{eq:supu_eps}. Lastly, for (iii), the uniqueness of the solution is due to Lemma \ref{lem:sgt-2}. 
Since we already have $u^+_\eps \geq u_\eps \geq u^-_{\varepsilon}$, we immediately get our conclusion from item {\it (ii)}.
\end{proof}

\begin{proof}[Proof of Corollary \ref{cor:homo}]
    It suffices to show that $\mathfrak{M}_-(u_\eps)=\varnothing$. Then Proposition~\ref{prop:ss}-(ii) yields the uniqueness of $u_\eps$ for \eqref{eq:HJ-ep}, and Theorem~\ref{thm:main} provides the convergence and its rate.
\medskip 

Under \ref{itm:A3'}, we have $\partial_u L > 0$ a.e., hence by Definition~\ref{defn:ord-M}, $\mathfrak{M}_-(u_\eps)=\varnothing$ for any solution $u_\eps$.

\medskip 
Under \ref{itm:A8}, we know that $\partial_u L(x,v,\theta)$ is non-decreasing in $\theta$ for fixed $(x,v) \in \eps\T^n \times \R^n$. If  $\mu \in \mathfrak{M}_-(u_\eps)$, then 
\[
\int_{\varepsilon\T^n \times \R^n} \partial_u L\left(\frac{x}{\varepsilon},v,u_\eps\right) d\mu = 0 \qquad\text{and}\qquad 
\int_{\varepsilon\T^n \times \R^n} L\left(\frac{x}{\varepsilon},v,u_\eps\right) d\mu = -c.
\]
Combining with the fact that $\partial_u L(x,v,\theta)$ is non-increasing and non-positive in $\theta$, we know that for any $\kappa \geq 0$ then
\begin{align*}
    0 = \int_{\varepsilon \T^n \times \R^n} \partial_u L\left(\frac{x}{\varepsilon},v,u_\eps \right) d\mu  
    \leq 
    \int_{\varepsilon \T^n \times \R^n} \partial_u L\left(\frac{x}{\varepsilon},v,u_\eps - \kappa\right) d\mu \leq 0.
\end{align*}
Therefore for any $\kappa>0$ then
\begin{align*}
    \int_{\varepsilon \T^n \times \R^n} \partial_u L\left(\frac{x}{\varepsilon},v,u_\eps - \kappa\right) d\mu = 0 
    \qquad\Longrightarrow\qquad 
    \int_{\varepsilon \T^n \times \R^n} L\left(\frac{x}{\varepsilon},v,u_\eps - k\right) d\mu = -c.
\end{align*}
Since $I(c) = \{\theta_0\}$, we can find a $\theta_1 < \theta_0$ such that $\overline H(0,\theta_1)=c_1<c$. By selecting $k$ sufficiently large satisfying $u_\eps - k \leq \theta_1$, we get
\begin{equation*}
    -c_1 \leq \int_{\varepsilon\T^n \times \R^n} L\left(\frac{x}{\varepsilon},v,\theta_1\right) d\mu \leq \int_{\varepsilon\T^n \times \R^n} L\left(\frac{x}{\varepsilon},v,u_\eps - k\right) d\mu = -c
\end{equation*}
which contradicts to $c_1<c$. Therefore, $\mathfrak{M}_-(u_\eps) = \varnothing$ and $u_\eps$ is the unique solution to \eqref{eq:HJ-ep}. 
\end{proof}

\section{Examples and discussions on the micro structure of $\cC$}\label{s3}
In this section, we prove some complementary conclusions for Theorem \ref{thm:main} and provide illustrative examples.
\begin{lem}\label{lem:I-single}
If ${\rm int}\;\cC\neq\emptyset$, then for almost every  $c\in\cC$, the associated  $I(c)$ is a singleton.
\end{lem}

\begin{proof}
Due to Lemma \ref{lem:erg-1st}, if ${\rm int}\;\cC\neq\emptyset$, then ${\rm int}\;\cC = (a,b)$ with $a,b \in \R\cup\{\pm \infty\}$. With out loss of generality, we assume that $(a,b )= (0,1)$ and other cases can be treated similarly. For any $c \in (0,1)$ such that $I(c)$ is not a singleton, there exists an interval $R_c$ with 
\begin{equation*}
    {\rm int} \; R_c \neq\emptyset    \qquad\text{and}\qquad \overline H(0,\theta)\equiv c\;\text{for all}\;\theta\in R_c. 
\end{equation*}
Therefore, we can always find a rational number $\theta_c \in R_c$. Since $\mathbb Q\cap (0,1)$ is countable and is of measure zero, we conclude that
the set of all such $c$ is a null set. 
\end{proof}

\begin{ex}\label{ex:sol1}
Let us consider
\begin{equation*}
	H(x, p, u) = \frac{1}{2}\vert p \vert^2 + (\cos(x) - 1)\Big(1 - \frac{1}{4\pi}\arctan(u)\Big).
\end{equation*}
We have ${\rm int}\;\cC=\emptyset$.
\end{ex}
\begin{proof}
It is not hard to check that the Hamiltonian defined above satisfies our assumption \ref{itm:A1}--\ref{itm:A3}. However, for any $\theta \in \R$, by the representation formula \eqref{eq:minmaxHbar} we have
\begin{equation*}
    \overline H(0,\theta) = \min_{\varphi\in C^1(\T^n)}\max_{x \in \T^n}
    \left[ 
        \frac{1}{2}|D\varphi(x)|^2 + 
        (\cos(x) - 1)\Big(1 - \frac{1}{4\pi}\arctan(\theta)\Big)
    \right].
\end{equation*}
Choose $\varphi\equiv 0$, we obtain $\overline{H}(0,\theta) \leq 0$. On the other hand, we have 
\begin{equation*}
    \overline H(0,\theta) \geq\max_{x \in \T^n}
    \left[ 
        (\cos(x) - 1)\Big(1 - \frac{1}{4\pi}\arctan(\theta)\Big)
    \right] = 0.
\end{equation*}
Thus $\cC = \{0\}$ and $I(0) = \R$.
\end{proof}

\begin{ex}\label{ex:sol2}
If $I(c)$ is not a singleton, equation \eqref{eq:HJ-overline} admits non-constant solutions in  ${\rm BUC}(\R^n,\R)$. Here is an example: Suppose 
\[
H(x,p,\theta) = H_1(x,p) + f(\theta),\quad (x,p,\theta)\in\R^3
\]
with $H_1(x,p) = \frac{1}{2}\vert p \vert^2 + (\cos(2\pi x) - 1)$ and $f(\cdot)$ being a smooth and monotone function with 
\begin{equation*}
    f(\theta) = 
    \begin{cases}
        \begin{aligned}
            &0, \quad &&\vert \theta \vert \leq 10\\
            & {\rm sgn}(\theta)(\theta - {\rm sgn}(\theta)\cdot 10)^2, \quad &&|\theta| > 10,
        \end{aligned}
    \end{cases}
\end{equation*}
   where ${\rm sgn}(\cdot)$ is the sign function. Then $\overline{H}(p,\theta) = \overline{H}_1(p) + f(\theta)$, where $\overline{H}_1(p)$ is the effective Hamiltonian of $H_1(x,p)$. Consequently, 
    \begin{equation*}
        I(0)=\{\theta\in\R~|~ f(\theta)=0\}=[-10,10]
    \end{equation*}
   is not a singleton.
   Define a function $u_0:\T(:=\R\slash\Z)\rightarrow\R$ by 
   \begin{equation*}
    u_0(x) = 
       \begin{cases}
           \begin{aligned}
               & \frac {4x}\pi,  && x \in \left[0, \frac{1}{2}\right],\\
                & \frac{4(1 - x)}{\pi}, &&  x \in \left[\frac{1}{2},1\right).
           \end{aligned}
       \end{cases}
   \end{equation*}
 It is not hard to check that $u_0$ is a viscosity solution of 
\[
\overline{H}(d_x u,u) = 0,\quad x\in\T.
\]
We can construct a sequence of smooth function $\{u_n\in C(\T, \R)\}_{n \in \Z_+}$ by
 $u_n(0)\equiv 0$ and 
$$
u_n'(x) = \left\{
\begin{aligned}
& \;\;\;\;\sqrt{2(\cos(4\pi nx) - 1)},\quad & x \in \left[0, \frac{1}{2}\right],\\
& -\sqrt{2(\cos(4\pi nx) - 1)}, \ & x \in \left[\frac{1}{2},1\right),
\end{aligned}
\right.
$$
and verified that $u_n$ solves 
\[
H(2nx,d_x u,u) = 0,\quad x\in\T.
\]
As we can see,  $u_n$ converges to $u_0$ uniformly as $n\rightarrow +\infty$, but the configuration space in the converging process is always $\T$, not $\T/n$.
\end{ex}

\section{Semilinear viscous Hamilton-Jacobi equations}\label{s4}

In this section we establish the existence of a couple of $(u,c)\in C^2(\T^n,\R)\times\R$ solving \eqref{eq:HJ-ep-2nd} under additional assumptions \ref{itm:A4}--\ref{itm:A7}. Since $H(x,\xi,\theta)$ is not assumed to be strictly monotone in $\theta$, a comparison principle is not available. Consequently, we adopt a vanishing discount procedure instead of Perron's method for the purpose of constructing a continuous viscosity solution, and then bootstrap its regularity. Before that, we should be aware that if $u\in C^2(\eps\T^n,\R)$ solves  \eqref{eq:HJ-ep-2nd} then
\[
c=\int_{\eps\T^n} H\left(\frac{x}{\varepsilon},du(x),u(x)\right)\,dx. 
\]

\noindent 
Firstly, we consider solutions to the following equation 
    \begin{equation}\label{eq:HJB2nd-th}
	   H(x,du,u) -\Delta u = c,\qquad\qquad  x\in \T^n
    \end{equation}
where the solutions are always assumed to be of the following sense:
\begin{defn}[Viscosity solutions (second-order)]\label{def:viscosity2nd} Let $u:\T^n\to \R$ and $H\in C(\T^n\times\R^n\times\R,\R)$.
\begin{itemize}
    \item[(i)] We say that $u$ is a viscosity subsolution to \eqref{eq:HJB2nd-th} if its upper semicontinuous envelope $u^*\in \mathrm{USC}(\T^n,\R)$ of $u$, is a viscosity subsolution, in the sense that, if $\varphi\in C^2(\T^n,\R)$ and $u^*-\varphi$ has a local max at $x_0$, then
    \begin{equation*}
	    H(x_0,d\varphi(x_0), u^*(x_0)) \leq c +  \Delta \varphi(x_0). 
    \end{equation*}
    \item[(ii)] We say that $u$ is a viscosity supersolution to \eqref{eq:HJB2nd-th} if its lower semicontinuous envelope $u_*\in \mathrm{LSC}(\T^n,\R)$ of $u$, is a viscosity supersolution, in the sense that, if $\varphi\in C^2(\T^n,\R)$ and $u_*-\varphi$ has a local min at $x_0$, then
    \begin{equation*}
	    H(x_0,d\varphi(x_0), u_*(x_0)) \geq c +  \Delta \varphi(x_0). 
    \end{equation*}
\end{itemize}
We say that $u$ is a viscosity solution if it is both a subsolution and a supersolution. 
\end{defn}

Now we state the Lipschitz estimate of solutions as follows (see \cite{Armstrong2015ViscosityEquations, barles_weak_1991, zhang_limit_2024} for related results).

\begin{prop}\label{prop:Bernstein}
 Assume that $H$ satisfies \ref{itm:A1}, \ref{itm:A4}-- \ref{itm:A7}. 
Let $u\in C(\T^n,\R)$ be a continuous viscosity solution to \eqref{eq:HJB2nd-th}. 
Then $u$ is Lipschitz with constant 
\begin{equation}
	\Vert du\Vert_{L^\infty(\T^n)} \leq C(\Lambda_0,\Lambda_\ell, M_0, M_\ell, \ell, \rho^*, c), \end{equation}
where $\ell = \Vert u\Vert_{L^\infty(\T^n)}$, $c$ is the ergodic constant and $\rho^*, \Lambda_0,\Lambda_\ell, M_0, M_\ell$ are related arguments involved in \ref{itm:A4}-- \ref{itm:A7}. Furthermore, we can show that $u\in C^2(\T^n,\R)$. 
\end{prop}

\begin{proof} We show that there exists $L>0$ such that for any $\hat{x}\in\T^n$,
$\limsup_{x\to\hat{x}} \frac{u(x)-u(\hat{x})}{|x-\hat{x}|} \leq L$. Assume by contradiction that for each $L>0$ there exists $\hat{x}\in\T^n$ with
\begin{equation}\label{eq:Lipschitz-contrary}
    \limsup_{x\to \hat{x}} \frac{u(x)- u(\hat{x})}{|x-\hat{x}|} > L. 
\end{equation}    
We will show that this leads to a contradiction. We define the auxiliary functional
\begin{equation*}
    \Phi(x,y) := u(x) - u(y) - L|x-y|
    , \qquad x,y \in \T^n. 
\end{equation*}
From \eqref{eq:Lipschitz-contrary}, we may assume that there exists $\gamma>0$ and a sequence $x_k\to \hat{x}$ such that 
\begin{equation*}
    \frac{u(x_k) - u(\hat{x})}{|x_k-\hat{x}|} > L + \gamma. 
\end{equation*}
We have
\begin{align}\label{eq:positive-max}
    \max_{x,y\in \T^n} \Phi(x,y) 
        \geq 
        \Phi(x_k, \hat{x})
        &= \Big( u(x_k) - u(\hat{x}) - L|x_k-\hat{x}| \Big) > 0
\end{align}
for $x_k$ sufficiently close to $\hat{x}$. 
By compactness of $\T^n$ and the continuity of $u$, there exist $x_0, y_0\in \T^n$ such that 
\begin{equation*}
    \max_{x,y\in \T^n} \Phi(x,y) = \Phi(x_0, y_0) > 0. 
\end{equation*}
Consequently, $x_0\neq y_0$. By Ishii's Lemma \cite[Lemma 3.2] {crandall_viscosity_1983} 
applying to $\varphi(x,y) = L|x-y|$, 
for every $\mu>0$ there exists matrices $X_\mu, Y_\mu \in \mathbb{S}^n$ such that 
\begin{equation}\label{eq:matrix-1}
\begin{cases}
\begin{aligned}
    & \left(\partial_x\varphi(x_0,y_0), X_\mu\right) \in \overline{J}_{\T^n}^{2,+} u(x_0), \\
    & \left(-\partial_y\varphi(x_0,y_0), Y_\mu\right) \in \overline{J}_{\T^n}^{2,-} u(y_0)
\end{aligned}
\end{cases} 
	\qquad \text{and}\qquad 
    \begin{pmatrix}
        X_\mu & 0 \\
        0 & -Y_\mu
    \end{pmatrix} \preceq A + \mu A^2,
\end{equation}    
where
\begin{equation*}
    A = \partial^2_{x,y}\varphi(x_0,y_0) = \frac{L}{|x_0-y_0|} \begin{pmatrix}
        \;\;Z &-Z\\
        -Z & \;\;Z
    \end{pmatrix}, \qquad Z = \mathbb{I}_n - \frac{x_0-x_0}{|x_0-y_0|} \otimes \frac{x_0-y_0}{|x_0-y_0|} . 
\end{equation*}
The subsolution test and supersolution test for $u$ give us
\begin{equation}\label{eq:viscTests}
	\begin{aligned}
	H\left(x_0, L\frac{x_0-y_0}{|x_0-y_0|}, u(x_0)\right) - \mathrm{Tr}(X_\mu) \leq c  
		\quad \text{and}\quad 	
     H\left(y_0, L\frac{x_0-y_0}{|x_0-y_0|}, v(y_0)\right) - \mathrm{Tr}(Y_\mu) \geq c . 
	\end{aligned}
\end{equation}
From \eqref{eq:matrix-1} we have
\begin{equation*}
    \begin{pmatrix}
        s \xi \\
        \xi 
    \end{pmatrix} ^T
    \begin{pmatrix}
        X_\mu & 0 \\
        0 &-Y_\mu 
    \end{pmatrix}
    \begin{pmatrix}
        s\xi\\
        \xi
    \end{pmatrix} 
    \leq 
    \begin{pmatrix}
        s \xi \\
        \xi 
    \end{pmatrix} ^T
    A
    \begin{pmatrix}
        s\xi\\
        \xi
    \end{pmatrix} 
    + 
    \mu 
    \begin{pmatrix}
        s \xi \\
        \xi 
    \end{pmatrix} ^T
    A^2
    \begin{pmatrix}
        s\xi\\
        \xi
    \end{pmatrix} 
\end{equation*}
for $s>1$ and $\xi\in \R^n$. We compute
\begin{align*}
     \begin{pmatrix}
        s \xi \\
        \xi 
    \end{pmatrix} ^T
    A
    \begin{pmatrix}
        s\xi\\
        \xi
    \end{pmatrix} = \frac{L}{|x_0-y_0|} (s-1)^2 \xi^T Z\xi .
\end{align*}
We note that 
\begin{equation*}
    \xi^T Z\xi = \xi^T \left(\mathbb{I}_n - \frac{x_0-y_0}{|x_0-y_0|}\cdot \left(\frac{x_0-y_0}{|x_0-y_0|}\right)^T\right) \xi  = |\xi|^2 -  \left|\xi^T \cdot \frac{x_0-y_0}{|x_0-y_0|}\right| ^2 \leq  |\xi|^2. 
\end{equation*}
Therefore
\begin{equation}\label{eq:estA1}
    \xi^T (s^2 X_\mu - Y_\mu) \xi \leq \frac{L(s-1)^2}{|x_0-y_0|}|\xi|^2 
    + \mu \begin{pmatrix}
        s \xi \\
        \xi 
    \end{pmatrix} ^T
    A^2
    \begin{pmatrix}
        s\xi\\
        \xi
    \end{pmatrix} . 
\end{equation}
Since $A^T = A$, we compute 
\begin{align*}
    \begin{pmatrix}
        s \xi \\
        \xi 
    \end{pmatrix} ^T
    A^2
    \begin{pmatrix}
        s\xi\\
        \xi
    \end{pmatrix}  = 
    \left(\begin{pmatrix}
        s \xi \\
        \xi 
    \end{pmatrix} ^T A \right) 
    \left(A \begin{pmatrix}
        s \xi \\
        \xi 
    \end{pmatrix} \right)  
    = 
    \left(A \begin{pmatrix}
        s \xi \\
        \xi 
    \end{pmatrix} \right) ^T
    \left(A \begin{pmatrix}
        s \xi \\
        \xi 
    \end{pmatrix} \right)  = 
    \left |
    A \begin{pmatrix}
        s \xi \\
        \xi 
    \end{pmatrix} 
    \right |^2.
\end{align*}
We compute
\begin{align}\label{eq:estA2}
     \left |
    A \begin{pmatrix}
        s \xi \\
        \xi 
    \end{pmatrix} 
    \right |^2 = \frac{2L^2(s-1)^2}{|x_0-y_0|^2} \Vert Z\xi\Vert^2 \leq \frac{2L^2(s-1)^2}{|x_0-y_0|^2} \Vert Z\Vert^2 |\xi|^2. 
\end{align}
We deduce from \eqref{eq:estA1} and \eqref{eq:estA2} that 
\begin{align*}
	\xi^T (s^2 X_\mu - Y_\mu) \xi
	\leq 	\left(	\frac{L(s-1)^2}{|x_0-y_0|^2} + \frac{2L^2(s-1)^2}{|x_0-y_0|^2} \Vert Z\Vert^2 \right)|\xi|^2.
\end{align*}
Therefore, from \eqref{eq:viscTests} we obtain
\begin{align*}
    &s^2  H\left(x_0, L\frac{x_0-y_0}{|x_0-y_0|}, u(x_0)\right)
    -
     H\left(y_0, L\frac{x_0-y_0}{|x_0-y_0|}, u(y_0)\right)  \\
     &\qquad\qquad\qquad\qquad 
     \leq 
               \mathrm{Tr}(s^2X_\mu - Y_\mu) + (s^2-1)c \\
     &\qquad\qquad\qquad\qquad
     \leq 
       \frac{nL(s-1)^2}{|x_0-y_0|} + n\mu \frac{2L^2(s-1)^2}{|x_0-y_0|^2} \Vert Z\Vert^2  + (s^2-1)c. 
\end{align*}
Let $s^2 = 1 + \beta|x_0-y_0|$ with $\beta>0$ to be selected later. We have 
\begin{equation*}
    (s-1)^2 \leq (s+1)^2(s-1)^2 \leq (s^2-1)^2 = \beta^2 |x_0-y_0|^2.     
\end{equation*}
We compute
\begin{align*}
    & (s^2-1)H\left(x_0, L\frac{x_0-y_0}{|x_0-y_0|}, u(x_0)\right) + 
    \left( 
    H\left(x_0, L\frac{x_0-y_0}{|x_0-y_0|}, u(x_0)\right) 
        - 
    H\left(y_0, L\frac{x_0-y_0}{|x_0-y_0|}, u(y_0)\right)
    \right)  \\
    &\qquad\qquad \qquad\qquad 
    \leq 
    nL\beta^2|x_0-y_0| + n\mu L^2\beta^2 \Vert Z\Vert^2 + \beta |x_0-y_0|c.
\end{align*}
For simplicity, let $\xi_0 = \frac{x_0-y_0}{|x_0-y_0|}$. We have
\begin{align*}
    & H(x_0, L\xi_0, u(x_0)) - H(y_0, L\xi_0, u(y_0))  \\
    & 
    \qquad = 
    H(x_0, L\xi_0, u(x_0)) - H(y_0, L\xi_0, u(x_0)) 
    + 
    H(y_0, L\xi_0, u(x_0)) - H(y_0, L\xi_0, u(y_0)) = I_1+ I_2.
\end{align*}
From \eqref{eq:positive-max} we have $u(x_0) > u(y_0)$, thus by \ref{itm:A5} we have $I_2\geq 0$.
On the other hand, from \ref{itm:A7} we have
\begin{align*}
    I_1 &= H(x_0, L\xi_0, u(x_0)) - H(y_0, L\xi_0, u(x_0)) \\
    &\qquad \qquad\qquad \qquad 
    \geq -(\Lambda_\ell L^m + M_\ell) |x_0-y_0|, \qquad \text{where}\qquad \ell  = \Vert u\Vert_{L^\infty(\T^n)}. 
\end{align*}
From \ref{itm:A5} we have $\big| H(x_0, L\xi_0, u(x_0)) - H(x_0, L\xi_0, 0) \big| 
    \leq 
    \rho^* \Vert u\Vert_{L^\infty(\T^n)}$.     
Therefore, by \ref{itm:A6} we have
\begin{equation*}
    H\big(x_0, L\xi_0, u(x_0)\big) 
        \geq 
    H\big(x_0, L\xi_0, 0\big) 
        -
    \rho^* \Vert u\Vert_{L^\infty(\T^n)}  \geq \Lambda_0|L|^m - M_0 - \rho^* \Vert u\Vert_{L^\infty(\T^n)}. 
\end{equation*}
We conclude that 
\begin{align*}
    & \beta|x_0-y_0|\Big( 
    \Lambda_0|L|^m - M_0 - \rho^* \Vert u\Vert_{L^\infty(\T^n)} 
    \Big) 
    -(\Lambda_\ell L^m + M_\ell) |x_0-y_0|  \\
    &\qquad\qquad\qquad 
    \leq 
    nL\beta^2|x_0-y_0| + n\mu L^2\beta^2 \Vert Z\Vert^2 + \beta |x_0-y_0||c|.
\end{align*}
Hence
\begin{align*}
    \beta\Big( 
    \Lambda_0|L|^m - M_0 - \rho^* \Vert u\Vert_{L^\infty(\T^n)} 
    \Big) 
    - (\Lambda_\ell L^m + M_\ell) 
    \leq 
    n L\beta^2 + \frac{n\mu L^2\beta^2 \Vert Z\Vert^2}{|x_0-y_0|} + \beta |c|.
\end{align*}
Sending $\mu\to 0^+$ we obtain 
\begin{align*}
    (\beta \Lambda_0 - \Lambda_\ell)L^m - \Big( \beta (M_0 + \rho^* \ell) +  M_\ell \Big) 
    \leq 
    n L\beta^2 + \beta |c|. 
\end{align*}
Choose $\beta = \frac{2\Lambda_\ell}{\Lambda_0}$, we deduce that 
\begin{align*}
	\Lambda_\ell L^m - \left(\frac{2\Lambda_\ell(M_0+\rho^*)\ell}{\Lambda_0}+M_\ell\right) \leq \frac{4\Lambda_\ell^2}{\Lambda_0^2} L + \frac{2\Lambda_\ell}{\Lambda_0} |c|. 
\end{align*}
In other words, we have
\begin{align*}
    L^m \leq \frac{4\Lambda_\ell}{\Lambda_0} L + \left(\frac{2|c|}{\Lambda_0} +  \frac{2(M_0+\rho^*)\ell}{\Lambda_0}+\frac{M_\ell}{\Lambda_\ell}\right) . 
\end{align*}
This implies that
\begin{align*}
    L\leq 2^{\frac{1}{m-1}}
    \left[
        \left(\frac{4\Lambda_\ell}{\Lambda_0}\right)^\frac{1}{m-1} 
        + 
        \left(\frac{2|c| + 2(M_0+\rho^*)\ell}{\Lambda_0}+\frac{M_\ell}{\Lambda_\ell}\right)^\frac{1}{m}
    \right]. 
\end{align*}
This contradicts the assumption \eqref{eq:Lipschitz-contrary}, yielding the desired Lipschitz property. 
\medskip

\noindent
As a consequence, $f = -H(x,du,u) + c \in L^p(\T^n,\R)$ for all $p\in [1,\infty]$. From $    -\Delta u = f$ in $\T^n$ we deduce, by standard elliptic regularity theory \cite{gilbarg_elliptic_2001} that $u\in W^{2,p}(\T^n,\R)$ for all $p\in (1,\infty)$, and thus $u\in C^{1,\alpha}(\T^n,\R)$ for all $\alpha \in (0,1)$. 
Since $\Vert du\Vert_{L^\infty(\T^n)}\leq C$ and $\Vert u\Vert_{L^\infty(\T^n)}\leq \ell$, 
we may modify $H(\cdot,\xi,u)$ for large $|\xi|$ and $|u|$ so that $H$ is Lipschitz with respect to $(\xi,u)$. Together with \ref{itm:A7}, we may assume without loss of generality that $H$ is Lipschitz with respect to all of its arguments. Then $f = -H(x,du,u) + c \in C^{0,\alpha}(\T^n,\R)$, and thus, by Schauder estimate \cite{gilbarg_elliptic_2001}, we obtain that $u\in C^{2,\alpha}(\T^n,\R)$ for all $\alpha \in (0,1)$. 
\end{proof}

Secondly, we prove that for any \(c \in \mathcal{C}_1\), defined as in \eqref{eq:C1def}, equation \eqref{eq:HJ-ep-2nd} is solvable for all \(\varepsilon > 0\). The fact that the definition of \(\mathcal{C}_1\) does not depend on \(\varepsilon\) calls for an explanation. For each \(\theta \in \mathbb{R}\) and \(p \in \mathbb{R}^n\), set
\[
\mathbf{H}_\theta(x,p) := H(x,p,\theta).
\]
By \cite[Theorem~5]{gomes_stochastic_2002}, there exists a unique ergodic constant \(\overline{H}(p,\theta) = \overline{\mathbf{H}}_\theta(p) \in \mathbb{R}\) such that \eqref{eq:erg-const-2nd}, i.e.,
\begin{equation*}
    H(x, p + d_x u,\theta) = \Dt u+\overline{H}(p, \theta),\quad x\in\T^n
\end{equation*}
admits a solution.

\begin{lem} \label{lem:ExisTheta} Assume \ref{itm:A1}--\ref{itm:A7}. For each $\theta\in \R$, there exists a classical solution $u_\eps\in C^2(\T^n,\R)$ to \eqref{eq:HJ-ep-2nd} with $c=\overline H(0,\theta)$. Moreover, $\{du_\eps\}_{\eps\in(0,1)}$ can be chosen to be uniformly bounded and uniformly Lipschitz.
\end{lem}

\begin{proof}
    By \cite[Theorem 5]{gomes_stochastic_2002}, there exists a unique classical solution $u_\theta \in C^2(\T^n,\R)$ with $u_\theta(0) = 0$ to $H(x,du_\theta, \theta) = \Delta u_\theta + \overline{H}(0,\theta)$ in $\T^n$. By \cite[Proposition 3]{gomes_stochastic_2002} we have
\begin{equation*}
    \inf_{x\in \T^n} H(x,0,\theta) \leq \overline{H}(0,\theta) \leq \sup_{x\in \T^n} H(x,0,\theta). 
\end{equation*}
Let $\om_\theta(x):=\eps u_\theta(x/\eps)$ for $x\in\eps\T^n$, then 
\begin{equation}\label{eq:sotheta}
    H\left(\frac{x}{\varepsilon}, d\om_\theta(x), \theta\right)=\overline{H}(0,\theta)+\eps\Dt\om_\theta(x),\quad x\in\eps\T^n.  
\end{equation}
Let $\om_\theta^{\pm}(x) := \theta + \om_\theta(x) \pm \Vert \om_\theta\Vert_{L^\infty(\T^n)}$ then $\om_\theta^- \leq \theta \leq \om_\theta^+$ in $\eps\T^n$. Consequently, For $\lambda>0$, we have 
\begin{align*}
    & \lambda \omega_\theta^+ + H\left(\frac{x}{\varepsilon}, d\omega_\theta^+, \omega_\theta^+\right) - \varepsilon \Delta \omega_\theta^+ \geq \lambda \theta + H\left(\frac{x}{\varepsilon}, d\omega_\theta, \theta\right) - \varepsilon \Delta \omega_\theta = \lambda\theta + \overline{H}(0,\theta), \quad x\in \varepsilon \T^n , \\
    & \lambda \omega_\theta^- + H\left(\frac{x}{\varepsilon}, d\omega_\theta^-, \omega_\theta^-\right) - \varepsilon \Delta \omega_\theta^- \leq \lambda \theta + H\left(\frac{x}{\varepsilon}, d\omega_\theta, \theta\right) - \varepsilon \Delta \omega_\theta = \lambda\theta + \overline{H}(0,\theta), \quad x\in \varepsilon \T^n. 
\end{align*}
Therefore, for each $\theta \in \R$, $\om_\theta^\pm$ are viscosity supersolution and subsolution of 
\begin{equation}\label{eq:HJB2nd-theta}
   \lambda u+ H\left(\frac{x}{\varepsilon},du, u\right) = \eps\Delta u + \lb\theta+\overline{H}(0,\theta), \qquad x\in \eps\T^n. 
\end{equation}
By the Perron method \cite[Theorem 4.1]{crandall_users_1992}, 
\begin{equation*}
    w_\lambda(x) := \sup 
    \left\lbrace 
        w:\eps\T^n\to \R,      w\;\text{is a viscosity subsolution to}\;\eqref{eq:HJB2nd-theta}, \om_\theta^-\leq w \leq \om_\theta^+ 
    \right\rbrace 
\end{equation*}
is the unique continuous viscosity solution to \eqref{eq:HJB2nd-theta}, since 
the comparison principle holds  for $\lb>0$. As we can see, $\{\om_\lb\}$ are  equi-bounded  for $\lb\in(0,1)$. It remains to show $\{\om_\lb\}_{\lb\in(0,1)}$ are also equi-Lipschitz. If so, then any accumulating function of $\om_\lb$ as $\lb\rightarrow 0_+$ has to be a solution of \eqref{eq:HJ-ep-2nd}, so we finish the proof. 
For this purpose, we take 
\begin{equation*}
    \widetilde{\omega}_\lambda(y):= \frac{1}{\varepsilon}\omega_\lambda(\varepsilon y), \qquad y\in\T^n    . 
\end{equation*}
Then $\widetilde{\omega}_\lambda$ is the solution to 
\begin{equation*}
    \lambda \varepsilon \widetilde{\omega}_\lambda(y) + H(y, d\widetilde{\omega}_\lambda(y), \varepsilon \widetilde{\omega}_\lambda(y))  = \Delta \widetilde{\omega}_\lambda(x) + \lambda\theta + \overline{H}(0,\theta), \qquad y\in \T^n. 
\end{equation*}
Define the new Hamiltonian 
\begin{equation*}
    \mathcal{H}_{\lambda,\varepsilon}(x,p,r):=  \lambda \varepsilon r + H(x, p, \varepsilon r),  \qquad (x,p,r)\in \T^n\times \R^n\times \R. 
\end{equation*}
By Proposition \ref{prop:Bernstein} with $c=c(\theta)+\lb\theta$, we obtain that $\{\widetilde{\omega}_\lb\}$ is uniformly Lipschitz, since the replaced Hamiltonian $\mathcal{H}_{\lambda, \varepsilon}$ satisfies \ref{itm:A5}--\ref{itm:A7} as the original Hamiltonian $H$ does. As a consequence, $\{\omega_\lambda:\lambda>0\}$ is uniformly Lipschitz, and the proof is complete. 
\end{proof}

\begin{rmk}
Under assumptions \ref{itm:A1}--\ref{itm:A7}, the conclusion of Theorem~\ref{thm:main-2}-(i) holds due to  Lemma \ref{lem:ExisTheta}. 
\end{rmk}

We conclude this section with a remark presenting an alternative, more direct proof of the existence of $u_\varepsilon$ for \eqref{eq:HJB2nd-th} without using vanishing discount procedure, under the additional assumption that $H(x,p,r)$ has superquadratic growth in the $p$-variable. This approach relies on the uniform H\"older continuity of solutions to \eqref{eq:HJB2nd-theta} under the superquadratic growth assumption on the Hamiltonian \cite{Armstrong2015ViscosityEquations, barles_short_2010, capuzzo_dolcetta_holder_2010}.

\begin{lem} Assume \ref{itm:A1}--\ref{itm:A5} and \ref{itm:A6},\ref{itm:A7} with $m>2$. For each $\theta\in \R$, there exists a classical solution $u_\varepsilon \in C^2(\T^n,\R)$ to \eqref{eq:HJB2nd-th} with $c = \overline{H}(0,\theta)$. 
\end{lem}
\begin{proof}[Sketch of a proof] From \cite{gomes_stochastic_2002} we obtain a subsolution $\omega^-_\theta$ and a supersolution $\omega^+_\theta$ of \eqref{eq:HJB2nd-th}, constructed via \eqref{eq:sotheta}. By the Perron method \cite[Theorem 4.1]{crandall_users_1992},
\begin{equation}\label{eq:subdefnition}
    w(x) := \sup 
    \left\lbrace 
        v:\varepsilon\T^n\to \R \;:\; v \text{ is a viscosity subsolution of } \eqref{eq:HJB2nd-th},\;
        \omega^-_\theta \le v \le \omega^+_\theta
    \right\rbrace
\end{equation}
is a viscosity solution of \eqref{eq:HJB2nd-th}. In the absence of a comparison principle, however, the continuity of this solution is not immediate. 

Since $m>2$ in \ref{itm:A6}--\ref{itm:A7}, it follows from \cite[Lemma 3.2]{Armstrong2015ViscosityEquations} (see also \cite{barles_short_2010, capuzzo_dolcetta_holder_2010}) that any subsolution in the admissible class of \eqref{eq:subdefnition} is H\"older continuous with a uniform H\"older estimate
\begin{equation*}
    |v(x) - v(y)| \le K |x-y|^\alpha,
    \qquad \alpha = \frac{m-2}{m-1},
\end{equation*}
where $K$ depends only on $n$, $\varepsilon$, and $H$ through \ref{itm:A1}--\ref{itm:A7}. Consequently, the Perron solution $w$ is also H\"older continuous and therefore belongs to \(C^2\) by Proposition~\ref{prop:Bernstein}.
\end{proof}

\section{Comparison principle for viscous HJ equations and the quatitative homogenization}\label{s5}

\subsection{Stochastic Mather measures and ergodic constants}
Before proving homogenization, we recall several basic results established in the works of Gomes~\cite{gomes_stochastic_2002}, Iturriaga and Morgado~\cite{iturriaga_stochastic_2005}, and Fleming and Soner~\cite{fleming_controlled_2006}.
The Hamiltonian in the following lemma is under the assumption \ref{itm:A1}, \ref{itm:A2}, \ref{itm:A4}. 

\begin{lem}[{\cite{fleming_controlled_2006,gomes_stochastic_2002,iturriaga_stochastic_2005}}]\label{lem:b-con} \quad 
\begin{enumerate}[(i)]
\item The ergodic problem
\begin{equation}\label{eq:HJ-2nd}
    H(x,p + du) = \Dt u + \overline H(p) \qquad\text{in}\;\T^n.    
\end{equation}
admits a unique viscosity solution (up to additive constants) $u$ and $\overline H(p)$ is the unique value that makes the equation solvable on $\T^n$. Moreover, if $v$ is a $C^2$ subsolution to \eqref{eq:HJ-2nd}, $v$ is actually a solution. 
\item The ergodic constant admits a representation
\begin{equation*}
    \overline H(p)=\min_{\varphi \in C^2(\T^n,\R)}\max_{x \in \T^n} H(x,p+d\varphi(x)) - \Dt\varphi(x).     
\end{equation*}
\item We denote by $\mathcal{H}_1$ the set of holonomic measure, namely probability measures on $T\T^n$ such that 
\begin{itemize}
\item $\int_{T\T^n} \big( \langle d_x \phi,v \rangle + \Dt \phi \big) d\mu = 0$ for any $\phi \in C^2(\T^n,\R)$;
\item $\int_{T\T^n} \vert v \vert d\mu < \infty$
\end{itemize}
Then
\begin{equation*}
    \overline H(p)=-\inf_{\mu\in\cH_1}\int \langle p,v\rangle +  L(x, - v)d\mu. 
\end{equation*}
The infimum can be obtained by the so called {\bf stochastic Mather measures}, which is unique. 

\item Let $\mu$ be the stochastic Mather measure, and let $\nu := \pi_1 \# \mu$ denote its projection onto $\T^n$, where $\pi_1 : \T^n \times \R^n \to \T^n$ is the canonical projection.
Then $d\nu = \vartheta(x)\,dx$ for some density $\vartheta$ with $\vartheta(x) > 0$ almost everywhere, and is a weak solution of
\begin{equation*}
    -{\rm div}\!\big(\vartheta(x)\, D_pH(x, p + Du(x))\big) - \Delta \vartheta(x) = 0.
\end{equation*}
For any $f \in C_c(\T^n \times \R^n,\R)$, we have
\begin{equation*}
    \int_{\T^n \times \R^n} f(x,v) d\mu = \int_{\T^n}f(x,D_p H(x, p + du(x)))\vartheta(x)dx.     
\end{equation*}

\item The Markov diffusion $dx = - D_pH(x, p + du(x))dt + \sqrt{2} dB_w(t)$, where $B_w$ stands for the standard $n$-dimensional Brownian motion, can define a measure $\mu_T$ in the way such that 
\begin{equation*}
    \int \varphi(x,v) d\mu_T = \frac{1}{T} \mathbb{E}\left[\int_0^T \varphi(x(t), D_pH(x, p + du(x)))dt\right], \quad \varphi \in C^2(\T^n \times \R^n,\R). 
\end{equation*}
Let $T\to\infty$, $\mu_T\rightharpoonup \mu_T$ where $\mu$ is the stochastic Mather measure. 

\item If $u_p$ is a solution to \eqref{eq:HJ-2nd}, then we have the following representation
\begin{equation*}
    u_p(x) = \inf \mathbb E\left[ u_p(x(t)) + \int_0^t \Big( \langle p, \mathfrak v \rangle + L(x(s), - \mathfrak v(x(s))) + \overline H(p) \Big) ds \right] 
\end{equation*}
where the infimum is taken in all measurable control $\mathfrak v$ with the Markov diffusion $dx = \mathfrak v(x)dt + \sqrt 2 dB_w$. Moreover, the infimum can be obtained by taking the control $\mathfrak v = -D_p H(x, p + du(x))$.

\item $\overline H(p)$ is convex in $p\in\R^n$ and $C^1$ smooth.
\end{enumerate}
\end{lem}

For any $\varphi\in C^1(\T^n,\R)$, we can define a Hamiltonian 
\[
{\bf H}_\varphi(x,p) := H(x,p,\varphi(x)),\quad \forall (x,p)\in T^*\T^n
\]
and above Lemma is still available. With this setup, we can first establish the following properties of $\overline H(0,\theta)$:

\begin{lem}\label{prop:c} Assume \ref{itm:A1}--\ref{itm:A7}. Then the mapping $\theta\mapsto \overline H(0,\theta)$ is continuous and nondecreasing. 
\end{lem}

\begin{proof} 
Due to \ref{itm:A3}, for any $a<b$ in $\R$ we have
\begin{align*}
    H(x,d_x u_b,a) - \Dt u_b \leq H(x,d_x u_b,b) - \Dt u_b\leq \overline{H}(0,b),\quad {\rm a.e. } \;x\in \T^n,
\end{align*}
where $u_b$ (resp. $u_a$) is a solution of \eqref{eq:erg-const-2nd} with $(p,\theta)=(0,b)$ (resp. $=(0,a)$), therefore, 
\begin{equation*}
    \overline{H}(0,a) \leq \overline{H}(0,b). 
\end{equation*}
For any  sequence $\R\ni a_n\to a$ as $ n\rightarrow+\infty$, we can find a sequence of solutions $\{u_n\}_{n\in\N}$. Due to \ref{itm:A2}, $\{u_n\}_{n\in\N}$ are uniformly Lipschitz, then uniformly bounded once we impose  $u_n(0)\equiv 0$ for any $n\in\N$. Suppose $c_*$ is an accumulating point of $\overline H(0,a_n)$ as $n\rightarrow+\infty$, then the associated subsequence of $\{u_n\}_{n\in\N}$ also has an accumulating function $\om$ as $n\rightarrow +\infty$, which is exactly a solution of  
\[
H(x, d_x\om, a) - \Dt \om=c_*,\quad x\in \T^n.
\]
Due to \cite{fathi_pde_2005}, such a $c_*\in\R$ is unique, so $c_*=\overline H(0,a)$ and $\lim_{n\rightarrow+\infty} \overline H(0,a_n)=\overline H(0,a)$ follows. 
\end{proof}

\subsection{Weak comparison principle}
Next, we present a lemma analogous to Lemma~\ref{lem:comparision} for the first-order case. For the proof of Theorem~\ref{thm:main-2}, it suffices to establish a comparison principle between $C^2$ subsolutions and $C^2$ supersolutions. In this setting, the argument is simpler than that of Lemma~\ref{lem:comparision}, thanks to the availability of the strong maximum principle for uniformly elliptic operators.

\begin{lem}\label{lem:cp-2nd}
Assume \ref{itm:A1}--\ref{itm:A4}, and $\varphi, \psi \in C^1(\T^n, \R)$, suppose that $u,w \in C^2(\T^n, \R)$ such that 
\begin{align}
    \mathbf{H}_\varphi(x,d_x u) &:= H(x,d_x u,\varphi(x)) \leq c + \Dt u   \label{eq:pde-u-sub} \\
    \mathbf{H}_\psi(x,d_x w) &:= H(x,d_x w,\psi(x)) \geq c + \Dt w. 
    \label{eq:pde-w-super}
\end{align}
Then one of the following two holds:
\begin{enumerate}[(i)]

\item The maximum of $u - v$ can be attained at a point $x_{max} \in \T^n$ where $\varphi(x_{max}) \leq \psi(x_{max})$.

\item For any $x \in \T^n$, we have $\varphi(x) - \psi(x) > 0$ and 
\[
u(x) - w(x) \equiv \max_{\T^n}(u - w).
\]
\end{enumerate}
\end{lem}

\begin{proof} 
Let 
\begin{align*}
    \mathbf{b}(x) &:= \int_0^1 D_pH\big(x,sd_xw(x) + (1-s)d_xu(x), \psi(x)\big) ds \\
    g(x) &:= H(x,d_xu(x), \varphi(x)) - H(x,d_xu(x), \psi(x))
\end{align*}
Then, by subtracting \eqref{eq:pde-w-super} from \eqref{eq:pde-u-sub}, we obtain
\begin{align}\label{eq:pde-elliptic}
    g(x) + \textbf{b}(x)\cdot Dz(x) - \Delta z(x) \leq 0 \qquad\text{in}\;\T^n, \qquad\text{where}\; z = u-w \in C^2(\T^n).
\end{align}
Let
\begin{equation*}
    E = \arg\max_{x\in \T^n} (u-v). 
\end{equation*}
Since $u,v$ are continuous, $E$ is a nonempty compact subset of $\T^n$. 
\begin{itemize}
    \item Case 1. There exists a point $x_0\in E$ such that $\varphi(x_0)\leq \psi(x_0)$. This implies (i). 
    \item Case 2. For all $x\in E$ we have $\varphi(x)>\psi(x)$. This implies that $g(x) \geq 0$ on $E$. For every $x\in E$, by continuity we can find $r_x>0$ such that $\varphi(x)>\psi(x) > 0$ on $B_r(x)$. By compactness, we can find a finite cover, and thus an open set $U\subset \T^n$ containing $E$ such that $\varphi(x)>\psi(x)>0$ on $U$, and thus $g(x)\geq 0$ on $U$. Therefore
    \begin{equation*}
        g(x) + \mathbf{b}(x)\cdot Dz(x) - \Delta z(x) \leq 0 \qquad\text{in}\; U. 
    \end{equation*}
    Since $E\subset U$, the function $z\in C^2(U)$ attains its maximum in the interior of $U$, and thus by the strong maximum principle \cite[Theorem 4]{evans_partial_2010}, $z$ is constant on any connected component of $U$ that intersects $E$. Since the constant value on each of these components is the same, each of these components is contained in $E$, and thus, consequently, $U = E$. This means $E$ is both open and closed in $\T^n$, we must have $E = \T^n$ as $E\neq \emptyset$. Therefore $z = \max_{\T^n} z = \mathrm{const}$ everywhere in $\T^n$, and $\varphi(x)>\psi(x)$ for every $x\in \T^n$. This is (ii).
\end{itemize}
\end{proof}

\begin{lem}\label{lem:sgt-3} Assume \ref{itm:A1}--\ref{itm:A4}. 
If
\beq\label{eq:sing-I(c)}
\max_{\theta \in I(c), x \in \T^n} \partial_u H(x,du_\theta(x),\theta) > 0
\eeq
where $u_\theta$ is the solution of 
	\beq\label{eq:hj-vis-theta}
	H(x,du_\theta(x), \theta) = \Delta u_\theta(x) + \overline H(0,\theta),\quad x \in \T^n
	\eeq
then $I(c)$ is a singleton.
\end{lem}

\begin{proof}
Assume $I(c)$ is not a singleton. Then for any $\theta \in I(c)$ strictly greater than $ \inf I(c)$, the solution $u_{\theta}$ of \eqref{eq:hj-vis-theta} associated with $c=\overline H(0,\theta)$ is $C^2-$smooth and unique up to additive constants (due to  Lemma \ref{lem:b-con}). From \ref{itm:A3},  we know that for any $\theta' \in I(c)$ with $\theta'< \theta$ then 
\begin{equation*}
    H(x, du_{\theta}, \theta') \leq H(x, du_{\theta}, \theta) = c + \Dt u_{\theta} = \overline{H}(0,\theta') + \Dt u_{\theta}.
\end{equation*}
By Lemma \ref{lem:b-con}, the subsolution $u_\theta$ is indeed a solution 
\begin{equation*}
    H(x, du_{\theta}, \theta') = c + \Dt u_{\theta} = \overline{H}(0,\theta') +  \Dt u_{\theta}.
\end{equation*}
In other words, for every $x\in \T^n$ we have $H(x,du_\theta(x), \theta') = H(x,du_\theta(x), \theta)$ for $\theta'<\theta$, hence 
\begin{equation}\label{eq:contra-I(c)}
    \partial_u H(x,du_\theta(x),\theta) = 0, \quad\forall\ x\in\T^n
\end{equation}
for any $\theta \in I(c)$ with $\theta > \inf I(c)$. This is a contradiction to \eqref{eq:sing-I(c)}, and thus $I(c)$ must be a singleton.
\end{proof}

For the $\eps$-dependent case,  we can use the same reduction as in subsection \ref{s2.2}: 
Notice that if $w(x) \in C^2(\eps\T^n,\R)$ is a solution (resp. subsolution) of 
\begin{equation*}
    {\bf H}_\varphi\left(\frac{x}{\varepsilon}, d_x w(x)\right) =  c + \eps \Dt w(x) , \quad x \in \eps\T^n
\end{equation*}
with $\varphi(x) \in C^1(\eps\T^n, \R)$, then $\vartheta(y) = \frac{1}{\eps} w(\eps y) \in C^2(\T^n)$ is a solution(resp. subsolution) of 
\[
{\bf H}_{\varphi \circ \eps}(y, d_y \vartheta(y)) = c + \Dt \vartheta(y), \quad y \in \T^n
\]
where $\varphi \circ \eps$  denotes by the function $\varphi(\eps x) \in C^2(\T^n,\R)$. In this case we can apply Lemma \ref{lem:cp-2nd} to ${\bf H}_{\varphi \circ \eps}(y, p)$ and $\vartheta(y)$, then all the conclusions working for $\vartheta(y)$ can be directly drawn for  $w(x)$ without additional difficulties. Therefore, we can get the following comparison principle for the solution $u_\eps$ of equation \eqref{eq:HJ-ep-2nd}:
\begin{prop}[Comparison principle]\label{prop:cp-vis}
Assume \ref{itm:A1}--\ref{itm:A7}.

\begin{enumerate}[(i)]

\item Suppose $u^1_\eps$ (resp. $u^2_\eps$) is a $C^2-$smooth solution (resp. subsolution) of \eqref{eq:HJ-ep-2nd} for fixed $c$, if there is a $x_0 \in \eps\T^n$ such that $u^2_\eps(x_0) > u^1_\eps(x_0)$, then $u^2_\eps - u^1_\eps$ is a contant on $\eps\T^n$ and $u^2_\eps$ is also a solution of \eqref{eq:HJ-ep-2nd}. Consequently, any two solutions of \eqref{eq:HJ-ep-2nd} differ by a constant.

\item For any $c_1,c_2 \in \cC_1$ with $c_1 > c_2$, any solution $u^1_\eps$ of \eqref{eq:HJ-ep-2nd} associated with value $c_1$ has to be strictly larger than any solution $u_\eps^2$ associated with $c_2$.

\item If we further assume \ref{itm:A3'}, then any \( C^2 \) subsolution of \eqref{eq:HJ-ep-2nd} cannot exceed the solution of \eqref{eq:HJ-ep-2nd}.

\item If $u_\eps$ is a solution of \eqref{eq:HJ-ep-2nd} and 
\begin{equation}\label{eq:uni-2nd}
    \max_{x \in \eps\T^n}\partial_u H\left(\frac{x}{\varepsilon},du_\eps(x),u_\eps(x)\right) > 0,
\end{equation}
then $u_\eps$ is the unique solution to \eqref{eq:HJ-ep-2nd} associated with $c$.
\end{enumerate}
\end{prop}

\begin{proof} 
    {\it (i).} Take $\varphi = u_\eps^2$ and $\psi = u^1_\eps$ in Lemma \ref{lem:cp-2nd}. Since we have $u^2_\eps(x_0) > u^1_\eps(x_0)$, we have 
    \begin{equation*}
        \max_{\eps\T^n}(\varphi-\psi)>0. 
    \end{equation*} 
    By Lemma \ref{lem:cp-2nd}-(ii) we obtain $u^2_\eps - u^1_\eps\equiv {\rm const}>0$ on $\eps\T^n$. Hence
    \begin{align*}
        & {\bf H}_{u^2_\eps}\left(\frac{x}{\varepsilon},du^2_\eps(x)\right) - \eps \Dt u^2_\eps(x) 
        \geq 
        {\bf H}_{u^1_\eps}\left(\frac{x}{\varepsilon},du^2_\eps(x)\right) - \eps \Dt u^2_\eps(x) \geq c,\quad \text{on\ }\eps\T^n \\
        & {\bf H}_{u^2_\eps}\left(\frac{x}{\varepsilon},du^2_\eps(x)\right) - \eps \Dt u^2_\eps(x) \leq c,\quad \text{on\ }\eps\T^n. 
    \end{align*}
Thus  $u^2_\eps$ is also a solution of \eqref{eq:HJ-ep-2nd}.
    \medskip 

    {\it (ii).} Since $c_1 > c_2$, so $u^2_\eps$  is a strict subsolution to \eqref{eq:HJ-ep-2nd} associated with $c_1$. 
    Previous arguments indicates $u^1_\eps \geq u^2_\eps$. If there exists $x_0 \in \T^n$ such that $u^1_\eps(x_0) = u^2_\eps(x_0)$, then $du^1_\eps(x_0) = du^2_\eps(x_0)$ and $\Dt(u^1_\eps - u^2_\eps) \geq 0$. Therefore, we have
    \begin{equation*}
        c_1 - c_2 = {\bf H}_{u^1_\eps}\left(\frac{x_0}{\varepsilon}, d u^1_\eps(x_0)\right) 
        - 
        {\bf H}_{u^2_\eps}\left(\frac{x_0}{\varepsilon}, d u^2_\eps(x_0)\right) - \eps \Dt(u^1_\eps - u^2_\eps)(x_0) \leq 0
    \end{equation*}
    which  contradicts $c_1 > c_2$. Hence $u^1_\eps > u^2_\eps$.
    \medskip 

    {\it (iii).} Due to (i), if the subsolution $u_\eps^1$ is larger than the solution $u_\eps^2$ at some point, then $u_\eps^1-u_\eps^2\equiv {\rm const}>0$, and by monotonicity \ref{itm:A3} we obtain 
    $\partial_u H\left(\frac{x}{\varepsilon},du_\eps(x),u_\eps^1(x)\right) = 0$. 
    That contradicts the assumption \ref{itm:A3'}.
    \medskip 

    {\it (iv).} Assume there is another solution $\hat{u}_\eps$ to \eqref{eq:HJ-ep-2nd}, then $\hat{u}_\eps$ and $u_\eps$ differ by a non-zero constant due to (i). So $\partial_u H\left(\frac{x}{\varepsilon},du_\eps(x),u_\eps(x)\right) = 0$ for any $x \in \eps\T^n$ which contradicts with \eqref{eq:uni-2nd}.
\end{proof}

We are now ready to prove Theorem \ref{thm:main-2}.

\begin{proof}[Proof of Theorem \ref{thm:main-2}]

Part {\it (i)} is already proved in Lemma \ref{lem:ExisTheta} and Lemma \ref{prop:c}. Part {\it (ii)} can be deduced from Proposition \ref{prop:cp-vis}-(i). \medskip

{\it (iii).} From Proposition \ref{prop:cp-vis} we know that any solution to \eqref{eq:HJ-ep-2nd} associated with the same value $c$ differ by constant. For any $\theta \in \R$, we define two functions
\begin{equation*}
    \vartheta_{\theta,\eps}^-(x) = \theta + \eps u_{\theta}\left(\frac{x}{\varepsilon}\right) - 2 \eps\Vert u_{\theta} \Vert_{L^\infty(\T^n)} 
    \quad \text{and}\quad 
    \vartheta_{\theta,\eps}^+(x) = \theta + \eps u_{\theta}\left(\frac{x}{\varepsilon}\right) + 2 \eps \Vert u_{\theta} \Vert_{L^\infty(\T^n)}, 
\end{equation*}
where $u_\theta$ is the solution to \eqref{eq:hj-vis-theta} associated with $\theta$. It is clear that for $\theta \in I(c)$ then
\begin{equation*}
    \Vert \vartheta_{\theta,\eps}^- -\theta\Vert_{L^\infty(\R^n)} \leq C\eps, \qquad 
    \Vert \vartheta_{\theta,\eps}^+ -\theta\Vert_{L^\infty(\R^n)} \leq C\eps
\end{equation*}
for some common constant $C = C(H,c)$. When $\theta_ -= \inf I(c) > -\infty$, let $u_\eps$ be a solution to \eqref{eq:HJ-ep-2nd}. We have
\begin{equation*}
    H_{\vartheta^-_{\theta, \varepsilon}}\left(\frac{x}{\varepsilon}, d\vartheta^-_{\theta,\varepsilon}(x)\right) \leq c + \varepsilon\Delta \vartheta^-_{\theta,\varepsilon}, 
    \qquad 
    H_{u_\varepsilon}\left(\frac{x}{\varepsilon}, du_\varepsilon\right) = c + \varepsilon\Delta u_\varepsilon.
\end{equation*}
\begin{itemize}
    \item By Lemma \ref{lem:cp-2nd} with $(u,\varphi) = (\vartheta_{\theta_-,\eps}^-, \vartheta_{\theta_-,\eps}^-)$ and $(w,\psi) = (u_\varepsilon, u_\varepsilon)$, we obtain that $\vartheta^-_{\theta_-, \varepsilon} \leq u_\varepsilon$ for any solution $u_\varepsilon$ of \eqref{eq:HJ-ep-2nd}, hence $u_\eps^-$ is well-defined, $u_\eps^- \geq \vartheta_{\theta,\eps}^-$ and $u_\varepsilon^-$ is a supersolution of \eqref{eq:HJ-ep-2nd}, and in this second-order case is indeed a solution of \eqref{eq:HJ-ep-2nd}.

    \item By Lemma \ref{lem:cp-2nd} with $(u,\varphi) = (u^-_\varepsilon,u^-_\varepsilon)$ and $(w,\psi) = (\vartheta^+_{\theta_-, \varepsilon}, \vartheta^+_{\theta_-, \varepsilon})$, we obtain that $u^-_\varepsilon \leq \vartheta^+_{\theta_-,\varepsilon}$. 
\end{itemize}
As a consequence we have $\Vert u_\eps^- - \inf I(c)\Vert_{L^\infty(\R^n)} \leq C\eps$. The case $\sup I(c) < \infty$ is similar. 
\medskip

{\it (iv).} The uniqueness of the solution of \eqref{eq:HJ-overline} is due to Lemma \ref{lem:sgt-2}. 
If $I(c) = \{\theta\}$ is a singleton then $\inf I(c) = \sup I(c) = \theta$, and furthermore, $u\equiv \theta$ is a solution to
\begin{equation*}
    H\left(\frac{x}{\varepsilon}, du, u\right) = c + \varepsilon\Delta u \qquad\text{for}\;x\in \varepsilon \T^n. 
\end{equation*}
Since $u^+_\eps \geq u_\eps \geq u^-_\eps$, we immediately get our conclusion from part (iii). 
\end{proof}

\begin{proof}[Proof of Corollary \ref{cor:homo2}]
Under the assumption \ref{itm:A3'}, we can directly get the conclusion by Proposition \ref{prop:cp-vis}-(iv). 
Under the assumption \ref{itm:A8}, it suffices to prove 
\begin{equation}\label{eq:temp-mono}
    \max_{x \in \eps\T^n}\partial_u H\left(\frac{x}{\varepsilon},du_\eps(x),u_\eps(x)\right) > 0. 
\end{equation}
If \eqref{eq:temp-mono} is true, then by Proposition \ref{prop:cp-vis}-(iv) we obtain the uniqueness of $u_\eps$. Let us assume the contrary of \eqref{eq:temp-mono}, i.e.,
\begin{align*}
    \max_{x \in \eps\T^n}\partial_u \left(\frac{x}{\varepsilon},du_\eps(x),u_\eps(x)\right) = 0.
\end{align*}
Then we have 
\begin{align*}
    \max_{x \in \eps\T^n}\partial_u H\left(\frac{x}{\varepsilon},du_\eps(x),f(x)\right) = 0 \qquad\text{for any}\;f \in C^1(\eps\T^n,\R)\;\text{such that}\; f\leq u_\varepsilon
\end{align*}
due to 
\begin{itemize}
    \item \ref{itm:A3}:  $\partial_u H(x,p,\cdot)$ is non-negative; 
    \item \ref{itm:A8}: $\partial_u H(x,p,\cdot)$ is non-decreasing in $\theta$ for any fixed $(x,p) \in \eps\T^n \times \R^n$), thanks to the convexity of $H$ in $\theta$. 
 \end{itemize}
 This means 
 \begin{equation*}
     H\left(\frac{x}{\varepsilon},du_\eps(x),f(x)\right) = H\left(\frac{x}{\varepsilon},du_\eps(x),u_\varepsilon(x)\right) \qquad\text{for any}\;f \in C^1(\eps\T^n,\R)\;\text{such that}\; f\leq u_\varepsilon. 
 \end{equation*}
 We deduce that, if $f\in C^1(\varepsilon\T^n;\R)$ with $f\leq u_\varepsilon$ then
\begin{equation}\label{eq:cor16-0}
    H\left(\frac{x}{\varepsilon},du_\eps(x),f(x)\right) = c + \eps \Dt u_\eps(x), \quad x \in \eps\T^n 
\end{equation}
Recall that $I(c)=\{\theta_0\}$ is a singleton, so for any $\theta_1 < \theta_0$ we have $\overline H(0,\theta_1) < c$. Let 
\begin{equation*}
    \theta_1 < \min\left \{\theta_0, \min_{\T^n} u^\varepsilon \right\}. 
\end{equation*}
By taking $f\equiv\theta_1$ in \eqref{eq:cor16-0} we obtain
\begin{align*}
    H\left(\frac{x}{\varepsilon},du_\eps(x),\theta_1\right) =c + \eps \Dt u_\eps(x), \quad x \in \eps\T^n.
\end{align*}
Let $w_\varepsilon(y) = \varepsilon u_\varepsilon\left(\varepsilon y\right)$ for $y\in \T^n$. We obtain
\begin{equation}\label{eq:theta1}
    H\left(y,dw_\eps(y),\theta_1\right) =c + \Dt u_\eps(x), \quad y \in \T^n.
\end{equation}
We recall that, the unique constant $c$ such that this equation admits a solution is $c = \overline{H}(0,\theta_1)$. This is a contradiction to the fact that $\overline H (0,\theta_1) < c$. Therefore, \eqref{eq:temp-mono} holds and we finish the proof.
\end{proof}

\noindent{\bf Acknowledgement.} The work of G. Liu and J. Zhang are supported by the National Key R\&D Program of China (No. 2022YFA1007500) and the National Natural Science Foundation of China (No. 12231010, 12571207). The authors thank H. Mitake and P. Ni for pointing out an error in an earlier version of the paper, and W. Jing and H. Tran for helpful comments on its presentation.


\end{document}